\documentclass[12pt,reqno]{amsart}

\usepackage[margin=1.2in]{geometry}

\usepackage[all, knot]{xy}

\usepackage{epsfig}
\usepackage{xspace}
\usepackage{layout}
\usepackage{latexsym}
\usepackage{amsmath}
\usepackage{stmaryrd}
\usepackage{amsthm}
\usepackage{amssymb}
\usepackage{amsfonts}
\usepackage{amsxtra}     
\usepackage{color}
\usepackage{verbatim}
\usepackage{longtable}
\usepackage{url}
\usepackage{tikz}
\usetikzlibrary{matrix,arrows,decorations.pathmorphing}

\usepackage{wrapfig}
\usepackage{lscape}
\usepackage{rotating}
\usepackage{placeins}

\usepackage{caption}
\usepackage[hypcap=true]{subcaption}

\usepackage{mathrsfs}

\usepackage[bookmarksnumbered, bookmarks, unicode=false, colorlinks=true,linkcolor=blue,citecolor=blue,filecolor=blue,urlcolor=blue]{hyperref}

\newcommand{\calh}{\mathcal{H}}
\newcommand{\calj}{\mathcal{J}}
\newcommand{\calu}{\mathcal{U}}

\newcommand{\co}{\colon\thinspace}

\newcounter{commentcounter}

\begin{document}

\newtheorem{thm}{Theorem}[section]
\newtheorem{conj}[thm]{Conjecture}
\newtheorem{lem}[thm]{Lemma}
\newtheorem{cor}[thm]{Corollary}
\newtheorem{prop}[thm]{Proposition}
\newtheorem{rem}[thm]{Remark}

\numberwithin{thm}{subsection}

\theoremstyle{definition}
\newtheorem{defn}[thm]{Definition}
\newtheorem{examp}[thm]{Example}
\newtheorem{construction}[thm]{Construction}
\newtheorem{notation}[thm]{Notation}
\newtheorem{rmk}[thm]{Remark}
\newtheorem{convention}[thm]{Convention}

\theoremstyle{remark}

\makeatletter
\renewcommand{\maketag@@@}[1]{\hbox{\m@th\normalsize\normalfont#1}}%
\makeatother

\renewcommand{\labelenumi}{(\roman{enumi})}
\renewcommand{\labelenumii}{(\alph{enumii})}

\renewcommand{\theenumi}{(\roman{enumi})}
\renewcommand{\theenumii}{(\alph{enumii})}

\def\square{\hfill ${\vcenter{\vbox{\hrule height.4pt \hbox{\vrule width.4pt
height7pt \kern7pt \vrule width.4pt} \hrule height.4pt}}}$}

\newenvironment{pf}{{\it Proof:}\quad}{\square \vskip 12pt}

\title[Voicing Transformations and Uniform Triadic Transformations]{Voicing Transformations and a Linear Representation of Uniform Triadic Transformations}

\author[Fiore and Noll]{Thomas M. Fiore and Thomas Noll}
\address{Thomas M. Fiore, Department of Mathematics and Statistics,
University of Michigan-Dearborn, 4901 Evergreen Road, Dearborn,
MI 48128 \newline  and \newline
NWF I - Mathematik,
Universit\"at Regensburg,
Universit\"atsstra{\ss}e 31,
93040 Regensburg,
Germany }
\email{tmfiore@umich.edu}
\urladdr{http://www-personal.umd.umich.edu/~tmfiore/}

\address{Thomas Noll \\ Escola Superior de M\'{u}sica de Catalunya \\
Departament de Teoria, Composici\'{o} i Direcci\'{o} \\
C. Padilla, 155 - Edifici L'Auditori \\
08013 Barcelona, Spain }
\email{thomas.mamuth@gmail.com}
\urladdr{http://user.cs.tu-berlin.de/~noll/}
\maketitle
\bigskip

\begin{abstract}
Motivated by analytical methods in mathematical music theory, we determine the structure of the subgroup $\calj$ of $GL(3,\mathbb{Z}_{12})$ generated by the three voicing reflections. As applications of our Structure Theorem, we determine the structure of the stabilizer $\calh$ in $\Sigma_3 \ltimes \calj$ of root position triads, and show that $\calh$ is a representation of Hook's uniform triadic transformations group $\calu$. We also determine the centralizer of $\calj$ in both $GL(3,\mathbb{Z}_{12})$ and the monoid $\text{Aff}(3,\mathbb{Z}_{12})$ of affine transformations, and recover a Lewinian duality for trichords containing a generator of $\mathbb{Z}_{12}$. We present a variety of musical examples, including the Wagner's hexatonic Grail motive and the diatonic falling fifths as cyclic orbits, an elaboration of our earlier work with Satyendra on Schoenberg, String Quartet in $D$ minor, op.~7, and an affine musical map of Joseph Schillinger. Finally, we observe, perhaps unexpectedly, that the retrograde inversion enchaining operation RICH (for arbitrary 3-tuples) belongs to the representation $\calh$. This allows a more economical description of a passage in Webern, Concerto for Nine Instruments, op. 24 in terms of a morphism of group actions.
\end{abstract}

\bigskip



\smallskip


\tableofcontents

\section{Introduction}

\subsection{Motivation for Transformational Approaches in Music Theory} \label{subsec:Motivation_for_Transformation} \leavevmode \smallskip

A driving motivation for the investigation of group actions on musical spaces is their application to the analysis of temporally ordered sequences of musical objects.\footnote{Another motivation is classification in terms of orbits and stabilisers.} Chord sequences, and in particular sequences of major and minor triads, constitute a central instance of this work. The musical objects under consideration are thereby conceived of as elements or ``points'' of an underlying space $S$ and the sequences become discrete trajectories $(s_0, s_1, \dots, s_n)$ within this space.

Major and minor triads in a chromatic 12-tone system can be encoded in various ways as the ``points'' of a 24-element set $S$. We briefly sketch encodings of consonant triads, since the rest of paper relies on this. As is usual in this area, we identify pitch classes with integers modulo 12 via the bijection $C \leftrightarrow 0$, $C\sharp \leftrightarrow 1$, \dots, and finally $B \leftrightarrow 11$. Consonant triads come in two types: major and minor. A {\it major triad} $\{r, r+4, r+7\}\subseteq \mathbb{Z}_{12}$ is said to have {\it root} $r$, {\it third} $r+4$, and {\it fifth} $r+7$.  The {\it letter name} of this major triad is the letter corresponding to the root $r$ under the aforementioned bijection. Similarly, a {\it minor triad} $\{r, r+3, r+7\}\subseteq \mathbb{Z}_{12}$ has {\it root} $r$, {\it third} $r+3$, and {\it fifth} $r+7$. The {\it letter name} of this minor triad is the letter corresponding to the root $r$. Major triads are indicated by capital letters, minor triads are indicated by lowercase letters. A {\it voicing} of a triad corresponds to a selected ordering encoded as a 3-tuple $(x,y,z)\in \mathbb{Z}_{12}^{\times 3}$. To summarize, one can encode a triad in three possible ways: as an unordered subset $\{x, y, z\}\subseteq\mathbb{Z}_{12}$, or as an ordered 3-tuple $(x,y,z) \in \mathbb{Z}_{12}^{\times 3}$ via a pre-selected unique voicing, or as a pair (root name, mode).

Returning to our motivation begun in the first paragraph, the ``transformational analyst'' judiciously selects group actions $G \times S \to S$ and seeks to interpret trajectories $(s_0, s_1, \dots, s_n)$ within $S$ via associated sequences $(g_1, \dots , g_n)$ of transformations $g_i \in G$ satisfying $g_i(s_{i-1}) = s_i$, as pictured in the network below.
$$
\xymatrix@C=3pc@R=3.5pc{*+=<1.7pc>[o][F-]{s_0} \ar[r]^{g_1} &  *+=<1.7pc>[o][F-]{s_1} \ar[r]^{g_2} &  \cdots \ar[r]^{g_n} & *+=<1.7pc>[o][F-]{s_n} }
$$

These networks are then themselves transformed and combined into larger networks that elucidate paradigmatic musical motions for the piece under investigation. The analyst thereby {\it presupposes that the transformations in question are defined on the entire space $S$}.

This global domain presupposition is crucial for the interpretation and evaluation of transformational analyses and needs to be understood in its radicality.  In \cite{LewinGMIT}, David Lewin (1993--2003) distinguished and compared two simply transitive actions on the consonant triads: one including global reflections (inversions), the other including contextual reflections (inversions).\footnote{For the moment it may suffice to acknowledge the fact that Lewin just distinguished the two actions. His mathematical insight, that these actions can even be understood as mutually dual actions of a dihedral group on the 24 consonant triads, shall be recalled and appreciated in Section~\ref{sec:Recollection_of_PLR}.} We illustrate the difference on consonant triads as unordered subsets of $\mathbb{Z}_{12}$ for the moment. The interpretation of the $e\flat$-minor triad $\{3,6,10\}$ as a mirror image of the $E\flat$-major triad $\{3,7,10\}$ (across the axis in the middle of 3 and 10) has two natural extensions to the set of all consonant triads. Either one can apply the global reflection operation $I_{3+10}(x)=-x+3+10$ throughout to all the consonant triads, or one may apply the contextual reflection operation {\it parallel} $P$ to all the consonant triads. In the definition of $P$, the {\it contextual} local mirror axis is selected so as to exchange the root and the fifth of the input chord, so that the input triad and the output triad overlap in the root and the fifth. The transformation $P$ has therefore been characterized as a {\it contextual inversion}. The precondition to ``know'' the effect of a transformation on the entire space can be satisfied in both cases through homogeneity assumptions about the underlying pitch class space. This allows, on the one hand, for the definition of pitch class inversions and transpositions. On the other hand, in the case of the contextual transformations, one could speak of an isotropy principle, i.e. a uniformity assumption about the collection of the conextual mirror axes. Two other such contextual inversions defined via common tone retention called {\it leading tone exchange} $L$ and {\it relative} $R$ are recalled in Section~\ref{sec:Recollection_of_PLR}.

For an example of transformational interpretations that illustrate the difference between global reflections, consider a {\it hexatonic cycle} of Cohn \cite{cohn1996}.
\begin{equation} \label{equ:hexatonic_cycle}
E\flat, e\flat, C\flat, b, G, g
\end{equation}
This progression is in measures 586--618 of Schubert, Piano Trio No. 2 in $E\flat$ Major, op. 100, 1st Movement, see the reduction by Cohn \cite[page 215]{cohn1999}.
There are (at least) two possible group-theoretic interpretations of the hexatonic cycle \eqref{equ:hexatonic_cycle}, one involving the alternating application of neo-Riemannian $P$ and $L$ operations, the other involving the componentwise global reflection operations $I_1$, $I_9$, and $I_5$, where $I_n(x):=-x+n$. The transformations $P$ and $L$ are described in more detail in Section~\ref{sec:Recollection_of_PLR}.
\begin{equation} \label{equ:PL_network_of_hexatonic_cycle}
\entrymodifiers={=<1.7pc>[o][F-]}
\xymatrix@C=3.5pc@R=3.5pc{E\flat \ar[r]^{P}  &  e\flat \ar[r]^L & C\flat \ar[r]^P & b \ar[r]^L & G \ar[r]^P & g \ar@/^2pc/[lllll]^L }
\end{equation}
\begin{equation} \label{equ:TI_network_of_hexatonic_cycle}
\entrymodifiers={=<1.7pc>[o][F-]}
\xymatrix@C=3.5pc@R=3.5pc{E\flat \ar[r]^{I_1}  &  e\flat \ar[r]^{I_9} & C\flat \ar[r]^{I_5} & b \ar[r]^{I_1} & G \ar[r]^{I_9} & g \ar@/^2pc/[lllll]^{I_5} }
\end{equation}
Notice that each occurrence of $P$ has a different reflection axis, as does each occurrence of $L$. What unifies these $P$-occurrences is the alignment of the contextual reflection axes perpendicularly to the fifth interval of the respective triads. Similarly, the $L$-occurrences have in common the perpendicular alignment of the reflection axes and the minor thirds, see Figure \ref{fig:Hexatonic_TraidsAndAxes} where the fifths are orange, the thirds are gray, and the reflection axes are dotted.

\begin{figure}[h]
  \centering
  \includegraphics[width=6in]{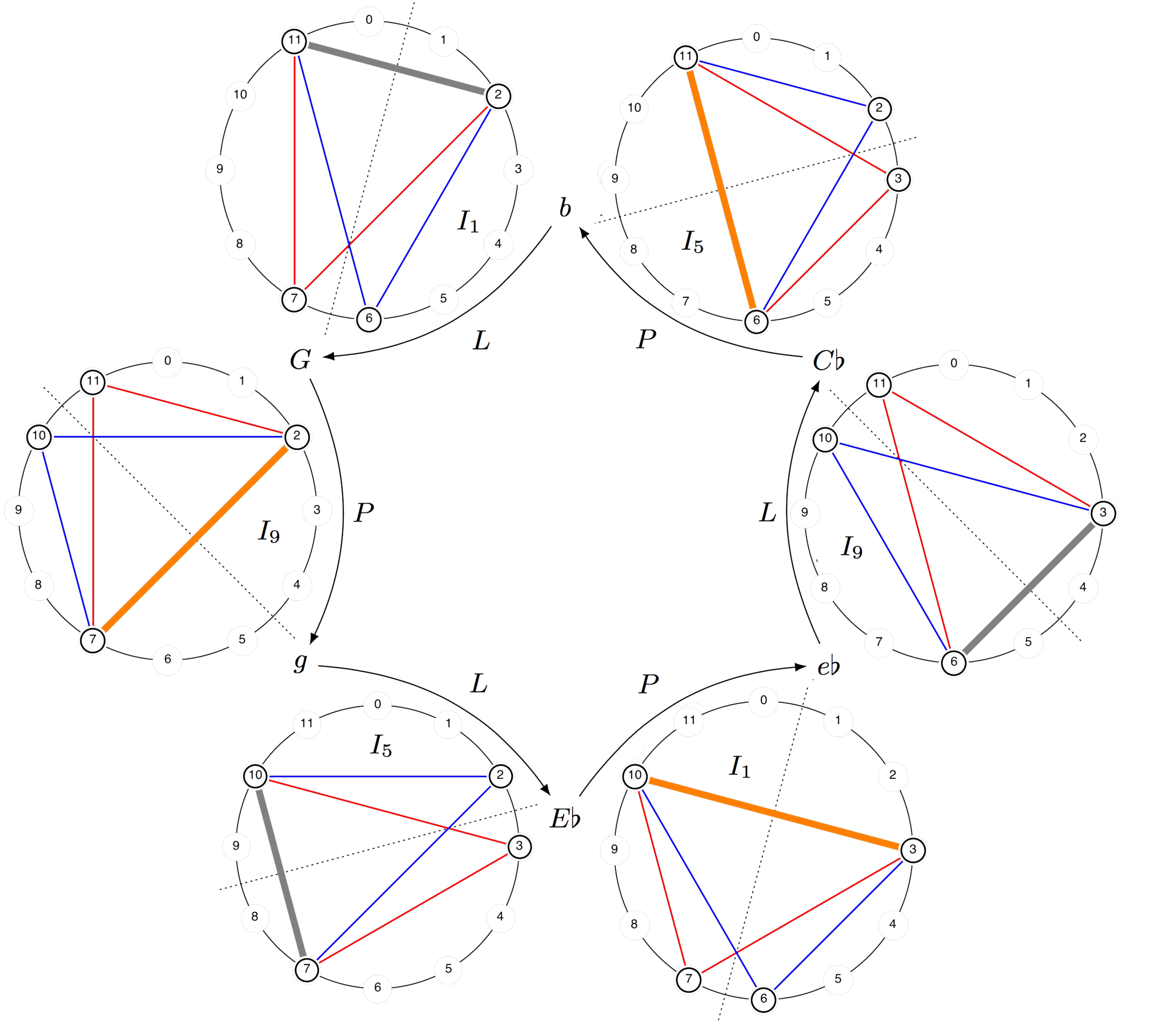}\\
  \caption{Two interpretations of the hextonic cycle $E\flat \mapsto e\flat \mapsto C\flat \mapsto b \mapsto G \mapsto g \mapsto E\flat$  (1) as a network of global reflections $I_1$, $I_5$ and $I_{11}$ and (2) as a network of contextual operations $P$ and $L$. The triangles represent consonant triads. The six clock face diagrams depict the occurrences of the global reflections $I_1$, $I_5$ and $I_{11}$, each of which occurs twice in the cycle. The colored fifths (orange) and minor thirds (gray) are always perpendicularly aligned to the respective reflection axis. These two types of contextual relations between triadic intervals and reflection axes characterize the transformations $P$ and $L$, respectively.}\label{fig:Hexatonic_TraidsAndAxes}
\end{figure}

The hexatonic cycle \eqref{equ:hexatonic_cycle}, with $C\flat$ enharmonically identified with $B$, has been reconsidered by Clampitt after Cohn in a hexatonic analysis \cite{clampittParsifal} of a particular variation of the Grail motive in {\it Parsifal}, Act 3, measures 1098--1100. The exact chord progression in \eqref{equ:hexatonic_cycle} and \eqref{equ:PL_network_of_hexatonic_cycle} is {\it not} in the Grail motive, rather the network is $PLP$ followed by $L$, followed by $PLP$. We realize this network via a single transformation in Section \ref{subsec:Grail_Motive}. Some mathematical background for \cite{clampittParsifal} was recently worked out in \cite{BerryFiore}.

What is the music-analytical intention behind the transformational approach? Critics sometimes see a kind of useless bookkeeping in the activity of labelling chord progressions through transformations. In fact, if the chosen group action is simply transitive, the analytical activity is suspiciously easygoing: there are as many musical objects as there are transformations available, and for any ordered pair $(s, s^\prime)$ of objects there is a unique transformation $g$ sending $s$ to $s^\prime = g(s)$. The transformational interpretation of the trajectories is thus completely determined as soon as a simply transitive group action is selected. To persuade the critics and oneself about the benefit of a transformational analysis, the music-theoretically crucial condition that the effect of a transformation must be ``known'' on the entire space must not be carelessly neglected in the mere bookkeeping of labels. Furthermore there should be an {\it economy of description} involved. The more chord successions exemplify the same transformation, the better.

For instance, let us reconsider the hexatonic cycle \eqref{equ:hexatonic_cycle} in light of ``economy of description'' to motivate Julian Hook's notion of uniform triadic transformation in \cite{hookUTT2002}. Notice that description~\eqref{equ:PL_network_of_hexatonic_cycle} is more economical than description~\eqref{equ:TI_network_of_hexatonic_cycle} because \eqref{equ:PL_network_of_hexatonic_cycle} only utilizes the {\it two} transformations $P$ and $L$, while \eqref{equ:TI_network_of_hexatonic_cycle} utilizes the {\it three} transformations $I_1$, $I_9$, and $I_5$. Alternating orbits under groups with two generators, such as the alternating $PL$-orbit in \eqref{equ:PL_network_of_hexatonic_cycle}, have been coined {\it flip-flop cycles} by John Clough in \cite{cloughFlipFlop}.
Is it possible to make an even more economical description of the hexatonic cycle with only a {\it single} transformation? More precisely, is it possible to define a single transformation which on major triads acts as $P$ and on minor triads acts as $L$ as in \eqref{equ:PL_network_of_hexatonic_cycle}? Hook answers affirmatively with the uniform triadic transformation $\langle -,0,8  \rangle$. The minus sign indicates the transformation reverses mode, the 0 indicates that a major input is not shifted before reversing mode, and the 8 indicates that a minor input is shifted by 8 before reversing mode. The flip-flop cycle \eqref{equ:PL_network_of_hexatonic_cycle} is thereby turned into an orbit of a cyclic transformation group.\footnote{John Clough remarks ``There is a tension between these two readings of a uniform flip-flop circle, one as a chain of paired involutions and the other as a chain of repeated one-way transformations. Which approach is preferable? The answer, I think, depends on one's objectives, and one's perceptions in a particular musical context.'' \cite[page 36]{cloughFlipFlop}}  The group $\calu$ of {\it uniform triadic transformations} with its 288 elements is much larger than the 24 element $PLR$-group, so there is a tradeoff between the economy of description and the size of the presupposed transformation group. We will recall uniform triadic transformations in Section~\ref{sec:A_Linear_Rep_of_the_UTTs} and argue that Hook's group $\calu$ occupies a quite natural position within the theoretical perspective of the present approach via a linear representation constructed from $\calj$ and the permutation $(1\;3)$.

David Lewin \cite{LewinGMIT} often positioned his analytical discourse in another 288-element group, namely in the group generated by the union of the atonal $T/I$-group and the contextual $PLR$-group.\footnote{The $T/I$-group and $PLR$-group commute with one another, are both of order 24, and have only two elements in common: the identity and $Q_6=T_6$. Consequently, the $T/I$-group and the $PLR$-group together generate a group of order $(24\times 24)/2 =288$.} Using these two competing kinds of transformations in tandem, he typically provides instructive arguments in the spirit of an economy of description principle within this chosen context.

So far, for consonant triads encoded as {\it ordered} 3-tuples, we have discussed two kinds of reflection (inversion): the global reflection operations $I_n$, and the contextual reflection operations $P$, $L$, and $R$. The present paper studies a third kind of reflection, called {\it voicing reflection}. In a voicing reflection, the local axis of reflection is determined by the tones in two pre-selected voices, for instance consider for the moment the voice reflection $W(x,y,z):=I_{x+z}(x,y,z)=(z,-y+x+z,x)$ determined by the bass and soprano voices. Let us consider the similarities and differences between the global reflection $I_1$, the contextual reflection $P$, and the voice reflection $W$. Recall that to compute the {\it parallel} contextual reflection $P$ on a consonant triad $(x,y,z)$, we look inside the chord to find the two tones $p$ and $q$ that span a perfect fifth, and then compute\footnote{Recall that $I_{p+q}\co\mathbb{Z}_{12} \to \mathbb{Z}_{12}$ is the unique global reflection that exchanges $p$ and $q$.} $P(x,y,z):=I_{p+q}(x,y,z)$.

On root position $E\flat$-major $(3,7,10)$, all three $I_1$, $P$, and $W$ coincide to give $e\flat$-minor with voicing $(10,6,3)$. On $E\flat$-major in any position, we see that $I_1$ and $P$ will coincide, but that they will differ from $W$ as soon as the first and last positions do not contain the pitch classes spanning the perfect fifth. For instance
$$P(3,10,7)=I_1(3,10,7)=(10,3,6)=\text{ $e\flat$-minor}$$
$$W(3,10,7)=I_{3+7}(3,10,7)=(7,0,3)=\text{ $c$-minor.}$$
Notice that the voicing does not effect the underlying unordered set of the $P$ output, but the voicing greatly effects the underlying output set of $W$. On inputs where the perfect fifth does not sum to 1, the transformations $I_1$ and $P$ differ. For instance,
$$I_1(4,8,11)=(9,5,2)$$
$$P(4,8,11)=I_{4+11}(4,8,11)=(11,7,4).$$
The transformations $W$ and $P$ agree only when the first and last positions contain the two pitch classes spanning a perfect fifth.

See Figure~\ref{fig:reflection_comparisons} for a comparison of the graphs of $I_1$, $P$, and $W$ on all 144 ordered consonant triads.

\begin{figure}[h]
    \centering
    \begin{subfigure}[b]{.45\textwidth}
        \centering
        \includegraphics[width=2.7in]{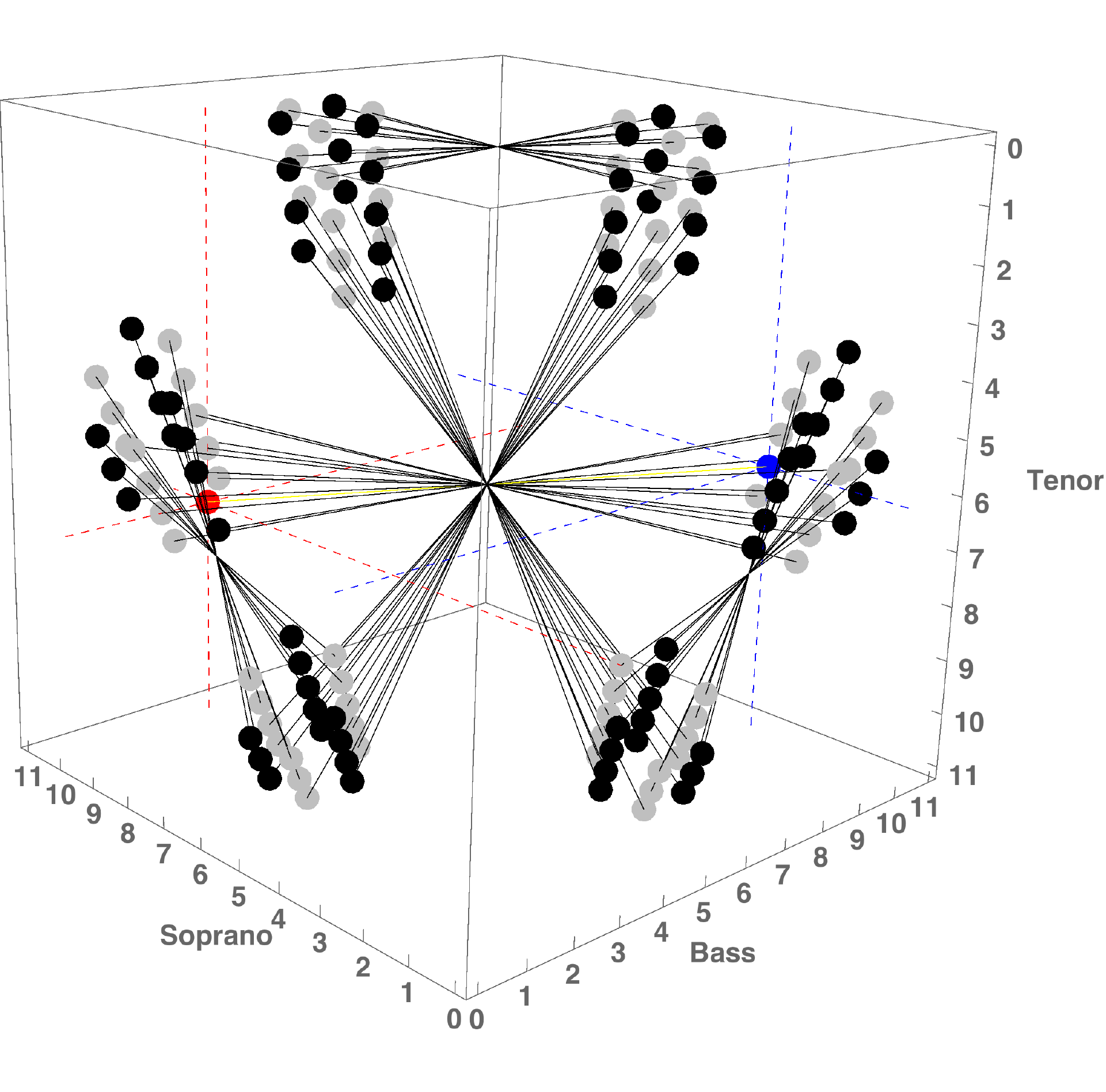}
        \caption{Global reflection $I_1$.}\label{}
    \end{subfigure}%
    \begin{subfigure}[b]{.45\textwidth}
        \centering
        \includegraphics[width=2.7in]{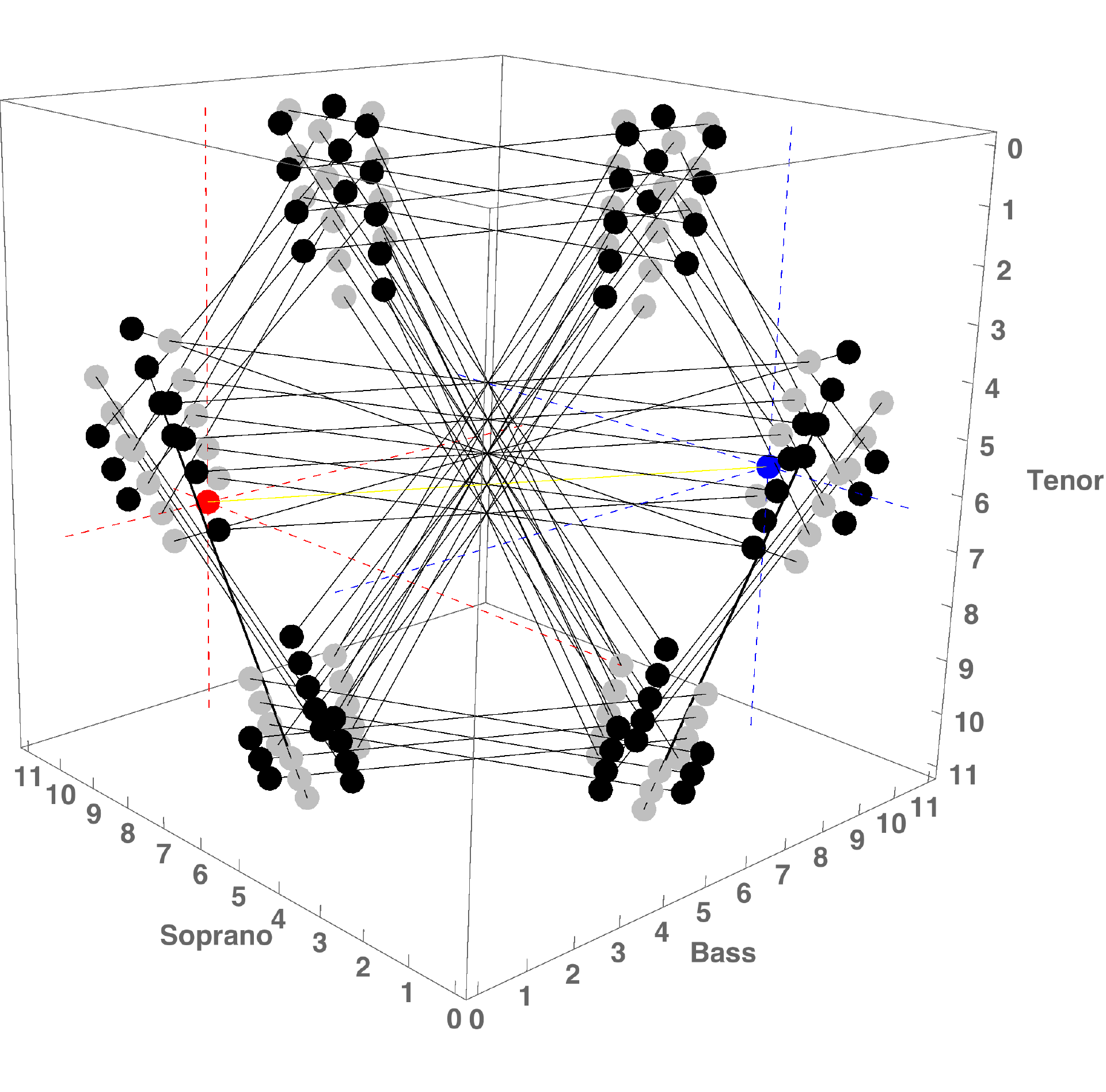}
        \caption{Contextual reflection $P$.}\label{}
    \end{subfigure}
   \begin{subfigure}[b]{1\textwidth}
        \centering
        \includegraphics[width=2.7in]{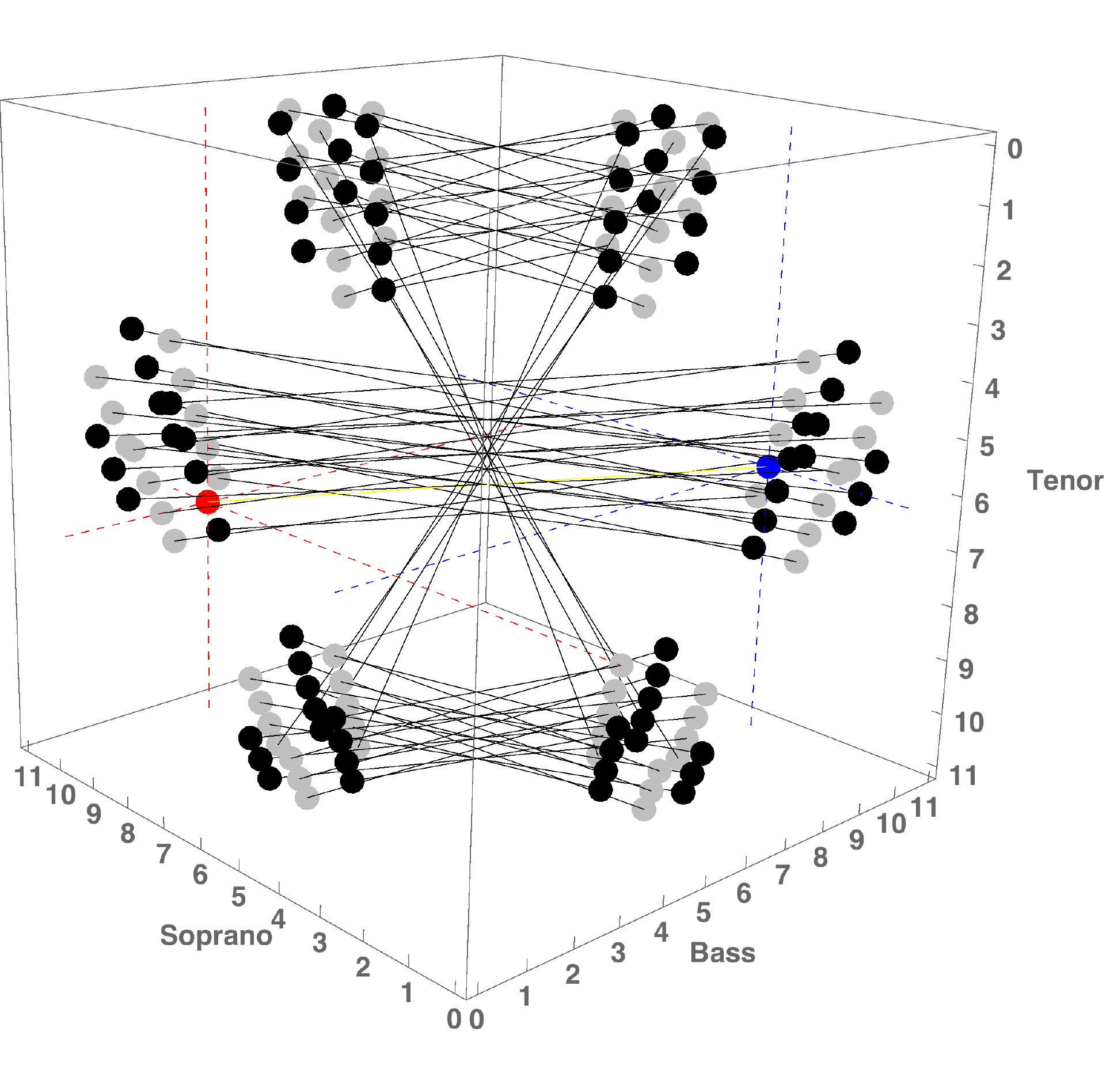}
        \caption{Voicing reflection $W$.}\label{}
    \end{subfigure}
   \caption{Comparison of global reflection $I_1$, contextual reflection $P$, and voicing reflection $W$. The cubes represent the complete voicing space $\mathbb{Z}_{12}^{\times 3}$. The 144 dots represent the voicings all consonant triads. Each black dot represents (a voicing of) a major triad, each gray dot represents (a voicing of) a minor triad. The red dot is $E\flat$ in root position $(3,7,10)$, and the blue dot is $e\flat$-minor in dualistic root position voicing $(10,6,3)$.} \label{fig:reflection_comparisons}
\end{figure}

\subsection{Motivational Problembeispiel} \label{subsec:MotivationalProblembeispiel} \leavevmode \smallskip

We motivate the mathematical questions and answers of the present paper with a new viewpoint on Straus' interpretation of Webern, Concerto for Nine Instruments, op. 24, Second Movement \cite[pages 57--61]{StrausContextualInversions}.
Using the usual encoding of pitch classes as integers modulo 12 and 3-pitch sequences as 3-tuples of such,
we horizontally list in Figure~\ref{fig:redo_of_Straus_on_Webern} the enchained sequences of measures 15--21 and 22--27 {\it in the order they occur in the score}.
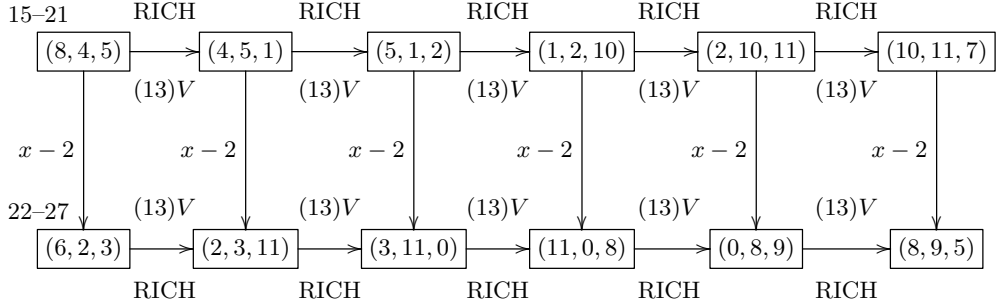
\begin{figure}[h]
\begin{center}
{\scriptsize
$$ \hspace{-.75in}
\renewcommand{\labelstyle}{\textstyle}
\entrymodifiers={+<2mm>[F-]}
\xymatrix@C=2pc@R=5pc{(8,4,5) \ar[r]_*!/_7pt/{(13)V}^*!/^11pt/{\text{RICH}} \ar[d]_{x-2} {}\save[]+<-0.6cm,.5cm>*\txt<8pc>{%
15--21} \restore & (4,5,1) \ar[r]^*!/^11pt/{\text{RICH}}_*!/_7pt/{(13)V} \ar[d]_{x-2}  & (5,1,2) \ar[r]_*!/_7pt/{(13)V}^*!/^11pt/{\text{RICH}} \ar[d]_{x-2} & (1,2,10) \ar[r]_*!/_7pt/{(13)V}^*!/^11pt/{\text{RICH}} \ar[d]_{x-2}  & (2,10,11) \ar[r]_*!/_7pt/{(13)V}^*!/^11pt/{\text{RICH}} \ar[d]_{x-2} & (10,11,7) \ar[d]_{x-2} \\
(6,2,3) \ar[r]^*!/^7pt/{(13)V}_*!/_11pt/{\text{RICH}} {}\save[]+<-0.6cm,.5cm>*\txt<8pc>{%
22--27} \restore & (2,3,11) \ar[r]^*!/^7pt/{(13)V}_*!/_11pt/{\text{RICH}} & (3,11,0) \ar[r]^*!/^7pt/{(13)V}_*!/_11pt/{\text{RICH}} & (11,0,8) \ar[r]^*!/^7pt/{(13)V}_*!/_11pt/{\text{RICH}} & (0,8,9) \ar[r]^*!/^7pt/{(13)V}_*!/_11pt/{\text{RICH}}  \restore & (8,9,5) }$$ }
\end{center}
\caption{Interpretation of Webern, Concerto for Nine Instruments, op. 24, Second Movement, measures 15--21 and 22--27. Each 3-note sequence is a permutation of some transposition/inversion of $\{0,1,4\}$. The rows are connected by retrograde inversion enchaining, abbreviated RICH. The pattern terminates after 6 sequences, but could continue to 8 sequences to traverse a complete cycle that exhausts the octatonic $\{1,2,4,5,7,8,10,11\}$ by Corollarly~6.2 of \cite{fiorenollsatyendraSchoenberg}.} \label{fig:redo_of_Straus_on_Webern}
\end{figure}
Consecutive interlocking 3-note sequences are connected by the RICH transformation, which is an acronym for {\it retrograde inversion enchaining} introduced by Lewin in \cite{LewinGMIT}, and first described in \cite{lewin1977}. The RICH transformation assigns to a pitch class sequence the reversed reflection whose first two numbers are the last two numbers of the input (compare Figure~\ref{fig:redo_of_Straus_on_Webern}).

Straus implements the RICH transformation of Figure~\ref{fig:redo_of_Straus_on_Webern} with two transformations he calls ``$L$'' and ``$P$'' that do not consider the ordering. The definitions of ``$L$'' and ``$P$'' are complex, as the 3-note sequences of Figure~\ref{fig:redo_of_Straus_on_Webern} are not consonant, and one must make some arbitrary conventions and refer to prime forms to make a definition analogous to neo-Riemannian $L$ and $P$. However, the analogy with neo-Riemannian $L$ and $P$ ends quickly, because the standard neo-Riemannian $PL$-cycle has an underlying hexatonic set covered by 6 cycle chords, whereas the two (incomplete) cycles of Figure~\ref{fig:redo_of_Straus_on_Webern} have underlying octatonic sets, each covered by 8 cycle chords. Simply replacing ``$L$'' by ``$R$'' to instead make an analogy with the neo-Riemannian octatonic $PR$-cycles does not solve the problem: no matter which labelling convention one chooses, a reordering of the 3-note sequences will lead to a clash with standard neo-Riemannian notations, as we can see in Table~\ref{table:restriction_of_J_to_orbits}.

In this paper we instead propose to use voice transformations $U$, $V$, and $W$ as defined in formulas \eqref{equ:P}, \eqref{equ:L}, \eqref{equ:R} of Section~\ref{subsec:Definition_of_J}, in combination with permutations such as $(13)$. Our usage of $(13)V$ for RICH in Figure~\ref{fig:redo_of_Straus_on_Webern} respects ordering, is immediately defined in terms of a straightforward formula without consonant connotations, and offers an economy of description with one transformation instead of two. Moreover, we know that $(13)V$ commutes with affine transformations such as $x-2$ by Proposition~\ref{prop:Centralizer_of_J_in_Aff(3,Z12)}.

We have arrived at the following questions, which we solve in the present paper.
\begin{itemize}
\item
What is the structure of the group $\calj$ generated by the voicing reflections $U$, $V$, $W$ {\it globally defined on} $\mathbb{Z}^{\times 3}_{12}$ by formulas \eqref{equ:P}, \eqref{equ:L}, \eqref{equ:R} as in Section~\ref{subsec:Definition_of_J}?
\item
Which affine endomorphisms of $\mathbb{Z}_{12}^{\times 3}$ commute with $U$, $V$, and $W$ and enable vertical morphisms such as in Figure~\ref{fig:redo_of_Straus_on_Webern}, or in Figure 10 of \cite{fiorenollsatyendraSchoenberg}?
\item
What is the structure of the subgroup of $GL(3, \mathbb{Z}_{12})$ generated by $\calj$ and the group $\Sigma_3$ of permutations on three letters?
\item
How can we describe flip-flop cycles more economically with a single transformation, as in the description of the $PL$-network in diagram \eqref{equ:PL_network_of_hexatonic_cycle} via the sole transformation $\langle -, 0, 8 \rangle$, or the description of the rows of Figure~\ref{fig:redo_of_Straus_on_Webern} via the sole transformation $(13)V$? Moreover, in doing this replacement, under which conditions can we retain an affine map as a morphism, as with $x-2$ in Figure~\ref{fig:redo_of_Straus_on_Webern}?
\item
How can we recover known duality theorems?
\item
How can we characterize a certain linear representation $\calh$ of Hook's uniform triadic transformations as a consequence of the Structure Theorem~\ref{thm:structure_of_J}?
\end{itemize}

\subsection{Outline of Contents} \leavevmode \smallskip

To keep the article self-contained, and to motivate the algebraic structures under investigation, we begin in Section~\ref{sec:Recollection_of_PLR} with a rapid review of the neo-Riemannian transformations $P$, $L$, $R$, their dihedral group, and their centralizer. We clarify the difference between two possible extensions of $P$, $L$, $R$: as contextual inversions via local permutation conjugation like in \cite{fiorenollsatyendraMCM2013,fiorenollsatyendraSchoenberg} and as discussed above, or as voicing reflections $W$, $V$, and $U$, which is the main subject of the present paper. In Section~\ref{sec:Recollection_of_PLR} we also foreshadow a normal form result and the connection to uniform triadic transformations by comparing $RL$ to $UV$ and writing it as a uniform triadic transformation.

Section~\ref{sec:The_Group_J} is an extensive study of the structure of the group $\calj$ generated by the voicing reflections $U$, $V$, and $W$. We first explain how $\calj$ restricts to six different $PLR$-groups on the various consonant orbits, and how each single voicing reflection restricts twice to three different $P$, $L$, $R$ transformations. Our first main result is Structure Theorem~\ref{thm:structure_of_J}, which specifies relations between the generators of $\calj$, gives a normal form for elements of $\calj$ together with {\it Schritt}-{\it Wechsel} type formulae, and identifies $\calj$ as a semi-direct product $\mathbb{Z}_2 \ltimes (\mathbb{Z}_{12} \times \mathbb{Z}_{12})$. Matrix representations for the normal form are in Remark~\ref{rem:matrix_representation_of_J_normal_form}. Then we go on to find the center of $\calj$ and the centralizers of $\calj$ in $GL(3,\mathbb{Z}_{12})$, $M(3,\mathbb{Z}_{12})$, $\text{Aff}(3,\mathbb{Z}_{12})$, and $\text{Aff}^\times(3,\mathbb{Z}_{12})$ in Propositions~\ref{prop:Center_of_J}, \ref{prop:Centralizer_of_J_in_GL3}, and \ref{prop:Centralizer_of_J_in_Aff(3,Z12)}. In Sections~\ref{subsec:Sigma3_J} and \ref{subsec:Triadic_Orbits} we bring permutations into the picture: Proposition~\ref{prop:Sigma3_J_is_SemiDirectProduct} identifies the subgroup of $GL(3, \mathbb{Z}_{12})$ generated by permutations $\Sigma_3$ and $\calj$ as the semi-direct product $\Sigma_3 \ltimes \calj$, while Section~\ref{subsec:Triadic_Orbits} distinguishes various relevant subgroups of $\Sigma_3 \ltimes \calj$ and important orbits. Section~\ref{subsec:Computer_Observations_On_Traces} presents the traces of normal form elements of $\Sigma_3 \ltimes \calj$, as determined by a computer.

Section~\ref{sec:Musical_Examples_And_Consequences} presents a variety of musical examples and consequences of the foregoing results on $\calj$ and $\Sigma_3$. Section~\ref{subsec:Grail_Motive} finds four elements of $\Sigma_3 \ltimes \calj$ that realize the flip-flop cycle $PLP$, $L$ from Wagner's Grail motive in {\it Parsifal}. In Section~\ref{subsec:Recalcitrant Viola} we revisit our earlier work with Ramon Satyendra \cite{fiorenollsatyendraSchoenberg} on Schoenberg, String Quartet in $D$ minor, op.~7 and include a viola passage via $(13)L$, map the result on a $PR$-cycle and a $PL$-cycle, and offer a more economical description in terms of $(13)V$. In Section~\ref{subsec:Affine_Morphisms} we continue with our work on Schoenberg to illustrate how Proposition~\ref{prop:Centralizer_of_J_in_Aff(3,Z12)} provides morphisms of generalized interval systems. In Section~\ref{subsec:Diatonic_Falling_Fifth} we remark that the present paper is valid not only to $\mathbb{Z}_{12}$ but for any $\mathbb{Z}_{n}$. In particular Theorem~\ref{thm:structure_of_J} applies to $\mathbb{Z}_7$, and we realize the diatonic falling fifths sequence via a linear transformation. Another interesting {\it mod} 7 example is in Section~\ref{subsec:SchillingerNetwork}, where we further specify Schillinger's $M_2$ map between Rimsky-Korsakov's {\it Hymn to the Sun} and Youmann's {\it Without a Song}. Section~\ref{subsec:Recovering_PLR-TI_Duality_and_FS} applies the commutativity Proposition~\ref{prop:Centralizer_of_J_in_Aff(3,Z12)} to recover the classical Lewinian duality of $PLR$ and $T/I$ as well as the special cases of pitch-class segment duality \cite{fioresatyendra2005} needed in the analysis \cite{fiorenollsatyendraSchoenberg} of Schoenberg, String Quartet in $D$ minor, op. 7.

In Section~\ref{sec:A_Linear_Rep_of_the_UTTs} we define and characterize a representation $\rho \co \calu \to GL(3, \mathbb{Z}_{12})$ of Hook's uniform triadic transformations group $\calu$. Essentially, for a uniform triadic transformation $u$, the linear map $\rho(u)$ is the unique linear extension of $u$ from root position consonant triads to all 3-tuples. Most of the results of Section~\ref{sec:A_Linear_Rep_of_the_UTTs} rely on the $\calj$ Structure Theorem~\ref{thm:structure_of_J}. We characterize the image of the embedding $\rho$ as the subgroup $\calh$ of $\Sigma_3 \ltimes \calj$ which preserves root position consonant triads, and investigate the structure of $\calh$ in terms of our results about $\calj$. In Proposition~\ref{prop:H_decomposition} we find a normal form associated to the decomposition $\calh=\calj^+ \;\bigsqcup \;(13) \calj^-$ where $\calj^+$ and $\calj^-$ are the mode-preserving respectively mode-reversing transformations in $\calj$. The image $\calh$ is generated by $(13)U$ and $(13)W$ as in Proposition~\ref{prop:H_two_generators}, which leads to a second normal form for $\calh$ in Proposition~\ref{prop:products_of_generators}. In Theorem~\ref{thm:H_is_semidirect_product} we prove that $\calh$ is the semi-direct product $\langle (13)W \rangle \ltimes \calj^+$. In Theorem~\ref{thm:towards_wreath_product} and Corollary~\ref{cor:New_Basis_Uniform_Formula} we select a new basis for $\calj^+$ in $\calh$ to prove that $\calh$ is a wreath product $\Sigma_2 \lbag \mathbb{Z}_{12}$ and to express the representation $\rho$ more directly in terms of uniform triadic transformation notation $\langle s, m ,n \rangle$. By this point it is clear that $\rho$ is an isomorphism onto its image.

In the Conclusion Section~\ref{sec:Conclusion} we revisit the {\it Problembeispiel} of Section~\ref{subsec:MotivationalProblembeispiel} and recall the more economical description in terms of the element RICH=$(13)V$ in $\calh$, also utilizing the centralizer results of Proposition~\ref{prop:Centralizer_of_J_in_Aff(3,Z12)}.

\subsection{Related Work of Rachel Hall} \leavevmode \smallskip

Rachel Hall's contribution \cite{Hall_LinearContextualTransformations} made initial advances in the study of voicing transformations. Her work is motivated\footnote{In Section~5.3 of \cite{Hall_LinearContextualTransformations}, Hall's mathematicizations of $P$, $L$, $R$ on consonant triads in dualistic root position are inspired by the equivalent descriptions of \cite{fioreMusicAndMathematics} and \cite{fioresatyendra2005}.} by \cite{fioresatyendra2005}. We acknowledge several aspects of her considerations as a groundwork for our own investigations. For reasons presented above we do not use Hall's proposed term {\it linear contextual transformations.}  Unlike our investigations of transformations on the discrete space $\mathbb{Z}_{12}^{\times 3}$, Rachel Hall \cite{Hall_LinearContextualTransformations} studies certain continuous linear transformations on $\mathbb{R}^n$ and makes a connection to the work of Callender, Quinn, Tymoczko \cite{CallenderQuinnTymoczkoScience, TymoczkoGeomMusChrdsScience}.
Concerning the mathematical findings, none of the theorems in our present paper are contained in \cite{Hall_LinearContextualTransformations}. Hall already noticed the representability of Hook's UTT-group in terms of voicing transformations in her Section~5.5. We elaborate this finding in several regards.

In \cite[Definition~3.1]{Hall_LinearContextualTransformations}, Hall defines the {\it linear contextual group} $\mathscr{C}^n$ to be the group of invertible linear maps $\mathbb{R}^n \to \mathbb{R}^n$ that commute with transposition and inversion, and induce well-defined linear transformations on the quotient $\mathbb{R}^n/(12\mathbb{Z})^n$. On page 112, she characterizes this group as the discrete group of invertible matrices with integer entries which fix the vector $\mathbf{1}=(1,1,\dots,1)$ of all 1's. $$\mathscr{C}^n=\{M \in GL(n,\mathbb{Z}) \; \vert \; M \mathbf{1}=\mathbf{1}\}$$
With respect to the basis
$$\mathbf{1},\; (-1, 1, 0, \dots , 0), \;(-1, 0, 1, 0, \dots , 0),\; \dots , \; (-1, 0, \dots , 0, 1),$$
the elements of $\mathscr{C}^n$ have a very nice form indicated in (3.2) on page 113 of her paper, which then leads to an isomorphism of $\mathscr{C}^n$ with the group of affine linear maps $\mathbb{Z}^{n-1} \to \mathbb{Z}^{n-1}$, see \cite[Theorem~3.1]{Hall_LinearContextualTransformations}. She proposes an encoding of $\mathscr{C}^n$ as elements $\langle \mathbf{a}, A \rangle$ with $\mathbf{a} \in \mathbb{Z}^{n-1}$ and $A \in GL(n-1,\mathbb{Z})$. In Section 5, she discusses how familiar transformations can be encoded in this way, such as contextual transpositions, retrograde, and uniform triadic transformations.

In Table~1 on page 117, Hall describes various subgroups of $\mathscr{C}^n$ and their semi-direct product structures.

\section{Recollection of the neo-Riemannian $PLR$-Group} \label{sec:Recollection_of_PLR}

The neo-Riemannian $PLR$-group is the algebraic point of departure for the present investigation. The original and authoritative source is David Lewin's pioneering book \cite{LewinGMIT}. A recent exposition of the neo-Riemannian $PLR$-group and its duality with the $T/I$-group can be found in \cite{cransfioresatyendra}. See also Fiore--Satyendra \cite{fioresatyendra2005} for an extension to $n$-tuples satisfying a tritone condition. Hook also treats the duality in \cite{hookUTT2002}.

The neo-Riemannian operations $P$, $L$, and $R$ are involutions on the set of 24 major and minor chords. The bijection $P$ assigns the {\it parallel major or minor chord}. The bijection $L$ is the {\it leading tone exchange}, which lowers the root of a major chord by a half step, and raises the fifth of a minor by half step. The bijection $R$ assigns the {\it relative major or minor chord}. For example, we have
$$P(C\text{-major}) = c\text{-minor} \hspace{.4in}
L(C\text{-major}) = e\text{-minor} \hspace{.4in}
R(C\text{-major}) = a\text{-minor}.$$
Musical motivation for these three transformations of consonant triads is the boundary conditions that input/output chords overlap in two pitches (or pitch classes) while the third pitch (or pitch class) moves by a minimal amount.

How can major/minor chords and these operations be mathematized so that no music-theoretical considerations are needed to compute them? We apply the methods of \cite{fioresatyendra2005}. Consider the set $S$ of certain 3-tuples with entries in $\mathbb{Z}_{12}$, namely the set $S$ consists of the 12 {\it major chords in root position}
$$(r,r+4,r+7) \hspace{.75in} r \in \mathbb{Z}_{12}$$
and the 12 {\it minor chords in reversed root position}
$$(r+7,r+3,r) \hspace{.75in} r \in \mathbb{Z}_{12}.$$
The bijections $P,L,R\co S \to S$, called {\it parallel}, {\it leading tone exchange}, and {\it relative} are formulaically defined on the set $S$ by
\begin{equation} \label{equ:P}
P(x,y,z)=(z,\;-y+x+z,\;x)\phantom{.}
\end{equation}
\begin{equation} \label{equ:L}
L(x,y,z)=(-x+y+z,\;z,\;y)\phantom{.}
\end{equation}
\begin{equation} \label{equ:R}
R(x,y,z)=(y,\;x,\;-z+x+y).
\end{equation}
The subgroup of $\text{Sym}(S)$ generated by the involutions $P$, $L$, and $R$ is called the {\it neo-Riemannian $PLR$-group} or simply  {\it $PLR$-group}. It acts simply transitively on $S$, it is dihedral of order 24, and is generated by $L$ and $R$ without $P$, in fact $P=R(LR)^3$.

The composite $RL$ has order 12, it adds 7 to each major triad and subtracts 7 from each minor triad, and preserves mode. As a {\it uniform triadic transformation on abstract triads}, $RL$ would be notated as $\langle +,7,-7 \rangle$, see Section~\ref{subsec:UTT_review}. We also see this uniform behavior when we consider $RL$ as the restriction of $UV$ in Section~\ref{subsec:Definition_of_J} to $S$ and consider the formula $UV(x,y,z)=(x,y,z)+(z-x)$ from Theorem~\ref{thm:structure_of_J}~\ref{item:action_of_normal_form}. Namely, on a major triad in $S$, we have $(z-x)=7$ but on a minor triad in $S$ we have $(z-x)=-7$. This highlights the importance of using ordered tuples in the algebraic formulation of $P$, $L$, and $R$.

Similarly, $LR$ has order 12, as it is the inverse of $RL$. The composites $PL$ and $LP$ have order 3, while $PR$ and $RP$ have order 4.

The involutions $P$, $L$, and $R$ on $S$ are {\it contextual inversions} in that they reflect chords across an axis that is determined by the input chord, rather than across a preselected axis for all chords. As a consequence, $P$, $L$, and $R$, so also the entire $PLR$-group, commute with the transposition and inversion operations $\mathbb{Z}_{12} \to \mathbb{Z}_{12}$ acting componentwise on 3-tuples.
$$T_n(x):=x+n \hspace{1in} I_n(x)=-x+n \hspace{1in} x,n\in \mathbb{Z}_{12}$$
These 24 transposition and inversion operations form the so-called {\it $T/I$-group}. The slash in the name $T/I$-group does {\it not} indicate any kind of quotient. {\it Lewinian duality} is the theorem that the $PLR$-group and the the $T/I$-group centralize each other in $\text{Sym}(S)$ and both act simply transitively on $S$.

The relationship between $P$, $L$, $R$ on $S$ and ordinary inversions $I_n$ is
\begin{equation} \label{equ:PI}
P(x,y,z)=I_{x+z}(x,y,z)\phantom{.}
\end{equation}
\begin{equation} \label{equ:LI}
L(x,y,z)=I_{y+z}(x,y,z)\phantom{.}
\end{equation}
\begin{equation} \label{equ:RI}
R(x,y,z)=I_{x+y}(x,y,z).
\end{equation}
In the left side of Figure~\ref{fig:Hexatonic_Cycle_Circle} we illustrated formulas \eqref{equ:PI} and \eqref{equ:LI} by revisiting the hexatonic cycle discussed in Section~\ref{subsec:Motivation_for_Transformation}, but now using the {\it ordered} triads in $S$.

\begin{figure}[h]
  \centering
  \hspace{-.45in}
  \includegraphics[width=6.5in]{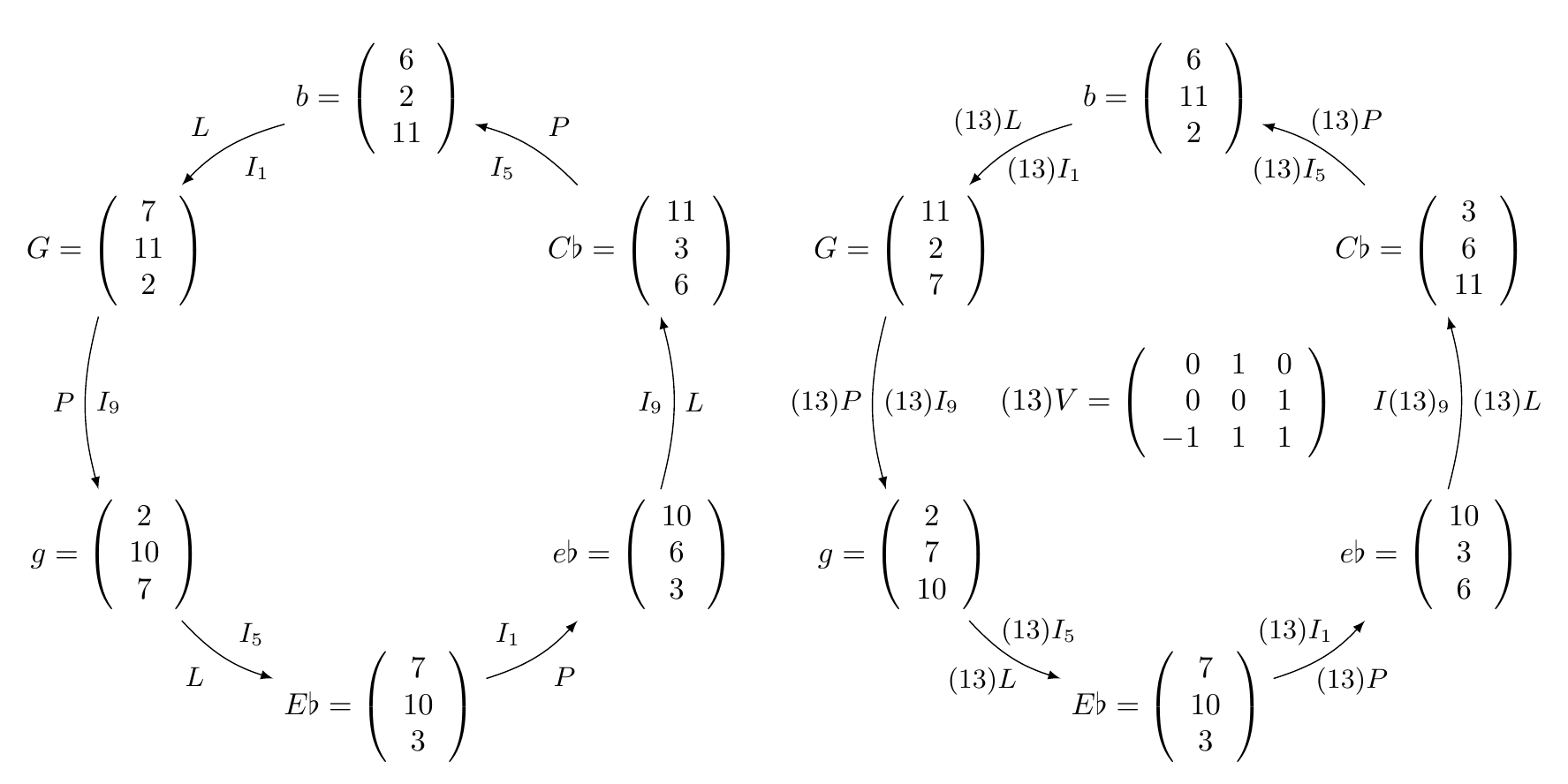}\\
  \caption{Left network: A revisitation of the hexatonic cycle \eqref{equ:hexatonic_cycle} with chords in dualistic root position. The two interpretations \eqref{equ:PL_network_of_hexatonic_cycle} and \eqref{equ:TI_network_of_hexatonic_cycle} are labelled and illustrate the relationships \eqref{equ:PI} and \eqref{equ:LI}. Right network: The global reflections $I_1, I_5, I_{11}$ and the contextual reflections $P, L$ act analogously on different voicings. But in contrast to the dualistic root position triads, these voicings form a cyclic orbit under the voicing transformation $(1 3) V$, that is, a RICH-cycle in the sense of Section \ref{subsec:MotivationalProblembeispiel}.}\label{fig:Hexatonic_Cycle_Circle}
\end{figure}

The present paper studies the {\it naive} extension of $P$, $L$, and $R$ defined on $S$ to linear functions $W$, $V$, and $U$ defined on all of $\mathbb{Z}_{12}^{\times 3}$. The formulas \eqref{equ:P}, \eqref{equ:L}, \eqref{equ:R}, or equivalently \eqref{equ:PI}, \eqref{equ:LI}, \eqref{equ:RI}, work equally well to define $W,V,U\co\mathbb{Z}_{12}^{\times 3} \to \mathbb{Z}_{12}^{\times 3}$ in Section~\ref{subsec:Definition_of_J}. We call $W$, $V$, and $U$ {\it naive} extensions because they are not proper extensions to contextually defined inversions. Namely, $W(0,4,7)=(7,3,0)$ but $W(4,7,0)=(0,9,4)$, that is $W$ acts on root position $C$-major as $P$, but $W$ acts on ``closed first inversion'' $C$-major as $R$, not exactly consistent behavior from a neo-Riemannian point of view. In anticipation of Section~\ref{sec:The_Group_J}, the right side of Figure~\ref{fig:Hexatonic_Cycle_Circle} illustrates the benefit of this point of view. The composition of the voicing transformation $V$ with the voice permutation $(13)$ allows the re-interpretation of a flip-flop cycle in terms of a purely cyclic orbit, and thereby it realizes a proposal by John Clough \cite{cloughFlipFlop} through a voicing transformation. With respect to contextual point of view, we proposed in \cite{fiorenollsatyendraSchoenberg} to extend $P$ to all permutations of major and minor chords via local conjugation. For instance, on $(1\;2\;3)S$ we define $P$ to be $(1\;2\;3)P(1\;2\;3)^{-1}$.  We prefer to call such extensions of $P$, $L$, and $R$ via local conjugation {\it contextual reflections} or {\it contextual inversions}, and to call $W$, $V$, and $U$ {\it voicing reflections}. The extension of $P$, $L$, and $R$ to contextual reflections via local conjugation is the same as the following. Consider a consonant triad $(x,y,z)$ in any order. To compute $P(x,y,z)$, we look inside the chord to find the two tones $p$ and $q$ that span a perfect fifth, and then compute $P(x,y,z):=I_{p+q}(x,y,z)$. To compute $L(x,y,z)$, we look inside the chord to find the two tones $p$ and $q$ that span a minor third, and then compute $L(x,y,z):=I_{p+q}(x,y,z)$. To compute $R(x,y,z)$, we look inside the chord to find the two tones $p$ and $q$ that span a major third, and then compute $R(x,y,z):=I_{p+q}(x,y,z)$.

The geometry behind the $PLR$-group, in the form of the {\it Tonnetz}, its dual, and related structures, has been studied and exposed in many places, for instance Catanzaro \cite{Catanzaro}, Clough \cite{CloughRudimentary}, Cohn \cite{cohn1997}, Douthett--Steinbach \cite{DouthettSteinbach}, Fiore--Crans--Satyendra \cite{cransfioresatyendra}, Gollin \cite{gollin}, and Waller \cite{Waller}. Extensions of the commutativity between $PLR$ and $T/I$ have been studied by Peck \cite{peckGeneralizedCommuting}.

\section{The Group $\mathcal{J}$ and its Extension $\Sigma_3 \ltimes \calj$} \label{sec:The_Group_J}

The review of the $PLR$-group in Section~\ref{sec:Recollection_of_PLR} motivates our definition of the group $\calj$ as  generated by the voicing reflections $U$, $V$, and $W$.

\subsection{Definition of $\calj$ via Generators $U$, $V$ and $W$} \label{subsec:Definition_of_J} \leavevmode \smallskip

We extend the formulas \eqref{equ:P}, \eqref{equ:L}, \eqref{equ:R} on major and minor triads to all of $\mathbb{Z}_{12}^{\times 3}$ to define linear automorphisms $U$, $V$, and $W$, each having the form of switching two coordinates, and adding their sum to the inverse of the third. In other words, $U$, $V$, and $W$ are voicing reflections, where the axis of reflection is determined by the input 3-tuple.
$$\begin{array}{rcc}
U(x,y,z):=J^{1,2}(x,y,z):=I_{x+y}(x,y,z)&=&(y,x,-z+x+y)\\
V(x,y,z):=J^{2,3}(x,y,z):=I_{y+z}(x,y,z)&=&(-x+y+z,z,y)\\
W(x,y,z):=J^{3,1}(x,y,z):=I_{z+x}(x,y,z)&=&(z,-y+x+z,x)
\end{array}$$

We identify $U$, $V$, and $W$ with their matrix representations as elements of $GL(3, \mathbb{Z}_{12})$.
$$U = \left( \begin{array}{ccc} 0 & 1 & 0 \\ 1 & 0 & 0 \\ 1 & 1 & {-1}\end{array} \right ), \;\;
V = \left( \begin{array}{ccc} {-1} & 1 & 1 \\ 0 & 0 & 1 \\ 0 & 1 & 0 \end{array}  \right ), \;\;
W = \left( \begin{array}{ccc} 0 & 0 & 1 \\ 1 & {-1} & 1 \\ 1 & 0 & 0 \end{array}  \right )$$
Let $\mathcal{J}$ be the subgroup of the general linear group $GL(3, \mathbb{Z}_{12})$ that is generated by $U$, $V$, and $W$. Notice that the determinant of each of $U$, $V$, and $W$ is 1, so that $\calj$ is actually a subgroup of the special linear group $SL(3, \mathbb{Z}_{12})$.

\subsection{Consonant Orbits of $\calj$} \leavevmode \smallskip

As a first step in understanding the group $\calj$, we may consider its action on the 144 arbitrarily ordered consonant triads in $\mathbb{Z}_{12}^{\times 3}$, find their orbits, and determine the restriction of $\calj$ to the individual orbits.

Consider the six $T/I$-orbits of the six reorderings of the $C$-major chord $(0,4,7)$. Each of the three generators $U$, $V$, $W$ of $\calj$ preserves these six $T/I$-orbits because $U$, $V$, $W$ act locally as $P$, $L$, or $R$ (exactly which of $U$, $V$, $W$ restricts to $P$, $L$, or $R$ depends on to which of the six $T/I$-orbits we are restricting). Each generator of $\calj$ restricts twice to $P$, $L$, and $R$ operations on the six $T/I$-orbits, as Table~\ref{table:restriction_of_J_to_orbits} indicates.

\begin{table}[h]
\centering
\begin{tabular}{|c|c|c|c|c|c|c|}
\hline
\multicolumn{7}{|c|}{Restriction of $\calj$-Generators to Six $T/I$-Orbits,}  \\
\multicolumn{7}{|c|}{Indicated by $C$-Major Representatives}\\
\hline \hline & closed & closed & closed & open & open & open \\
       & root pos & first inv & second inv & root pos & first inv & second inv \\
 & $\left( \begin{array}{c} 0 \\ 4 \\ 7 \end{array} \right)$ & $\left(\begin{array}{c} 4 \\ 7 \\ 0 \end{array}\right)$ & $\left(\begin{array}{c} 7 \\ 0 \\ 4  \end{array}\right)$ & $\left(\begin{array}{c} 0 \\ 7 \\ 4 \end{array}\right)$ & $\left(\begin{array}{c} 4 \\ 0 \\ 7  \end{array}\right)$ & $\left(\begin{array}{c} 7 \\ 4 \\ 0  \end{array}\right)$ \\
\hline $U$ & $R$ & $L$ & $P$ & $P$ & $R$ & $L$ \\
\hline $V$ & $L$ & $P$ & $R$ & $L$ & $P$ & $R$ \\
\hline $W$ & $P$ & $R$ & $L$ & $R$ & $L$ & $P$ \\
\hline
\end{tabular}
\bigskip
\caption{} \label{table:restriction_of_J_to_orbits}
\end{table}

Thus, the restriction of $\calj$ to the $T/I$-orbit of any reordering of $(0,4,7)$ is a copy of the $PLR$-group and acts simply transitively, so we see that the consonant orbits of $\calj$ are precisely the $T/I$-orbits of the six permutations of $(0,4,7)$.
However, the restriction of $\calj$ to any of its consonant orbits has a nontrivial kernel, for instance $(WV)^3$ is $(PL)^3$ on the orbit of $(0,4,7)$, so the identity there, although $(WV)^3$ is itself a nontrivial element of $\calj$ (for instance, on the orbit of $(4,7,0)$ it is $(RP)^3$ which is not the identity). The kernel of the restriction $r\co \calj \to \text{Sym}(S)$ has 12 elements, since
$$\vert\text{ker}\;r\vert=\frac{\vert \calj \vert}{\vert \text{im} \; r \vert}=\frac{288}{24}=12.$$

The restriction of $\mathcal{J}$ to the six $T/I$-orbits of the permutations of $(0,4,7)$ is a group homomorphism $comp\co \mathcal{J} \to (PLR\text{-group})^{\times 6}$ which on generators is given by the bottom three rows of Table~\ref{table:restriction_of_J_to_orbits}.

\subsection{The Structure of $\mathcal{J}$} \leavevmode \smallskip

\begin{convention} \label{conv:adding_constant_to_3tuple}
In the interest of readability, to indicate the addition of a constant $c\in \mathbb{Z}_{12}$ to each component of $(x,y,z)\in \mathbb{Z}_{12}^{\times 3}$, we write $$(x,y,z)+c:=(x+c,y+c,z+c).$$
\end{convention}

\begin{thm}[Structure of the Group $\calj$] Consider the subgroup $\calj$ of $SL(3, \mathbb{Z}_{12})$ generated by $U$, $V$, and $W$ as in Section~\ref{subsec:Definition_of_J}.
\label{thm:structure_of_J}
\begin{enumerate}
\item
The generators $U$, $V$, and $W$ satisfy the following relations.
\begin{enumerate}
\item  \label{thm:structure_of_J:U_V_W_have_order_2}
Each of $U$, $V$, and $W$ has order 2.
\item \label{thm:structure_of_J:UV_and_UW_have_order_12}
Both composites $UV$ and $UW$ have order 12.
\item \label{thm:structure_of_J:UVW_has_order_2}
The composite $UVW$ has order 2.
\item
The composites $UV$ and $UW$ commute.
\item \label{U-conjugation_is_inversion}
The $U$-conjugation of $(UV)^m$ and $(UW)^n$  is inversion.
$$U^{-1}(UV)^mU=(UV)^{-m} \hspace{1in} U^{-1}(UW)^nU=(UW)^{-n}$$
\end{enumerate}
\item \label{thm:structure_of_J:semi-direct_product_form}
Every element of $\calj$ can be written uniquely in the form
\begin{equation} \label{equ:normal_form}
U^k(UV)^m(UW)^n
\end{equation}
where $k=0,1$ and $m,n=0, 1, \dots, 11$.
\item
The group $\mathcal{J}$ has order $288$.
\item
The group $\mathcal{J}$ is the internal semi-direct product $\langle U \rangle \ltimes \langle UV, UW\rangle$, so is
isomorphic to the semi-direct product $\mathbb{Z}_{2} \ltimes (\mathbb{Z}_{12} \times \mathbb{Z}_{12})$ where $\mathbb{Z}_2$ acts on $\mathbb{Z}_{12} \times \mathbb{Z}_{12}$ via additive inversion.
\item \label{item:action_of_normal_form}
The elements of $\calj$ in the normal form of \eqref{equ:normal_form} act as follows.
$$(UV)^m(UW)^n(x,\;y,\;z)=\big(x,\;y,\;z\big)+ m(z-x) + n(z-y)$$
$$U(UV)^m(UW)^n(x,\;y,\;z)=U(x,\;y,\;z)+ m(z-x) + n(z-y)$$
\item
For completeness, we also observe
$$(VW)^j(x,\;y,\;z)=\big(x,\;y,\;z\big)+ j(x-y)$$
and $VW$ has order 12.
\end{enumerate}
\end{thm}
\begin{proof}
\leavevmode
\begin{enumerate}
\item \label{i}
\begin{enumerate}
\item
The computations $U^2=V^2=W^2=\text{Id}$ are straightforward.
\item
To see that $UV$ has order 12, notice that $UV(x,y,z)$ is the addition of $(z-x)$ to each component,
\begin{equation} \label{equ:UV}
UV(x,y,z)=\big(x,\;y,\;z\big)+(z-x)
\end{equation}
(here we use Convention~\ref{conv:adding_constant_to_3tuple}).
Thus, an application of $UV$ to the outcome of \eqref{equ:UV} will similarly add the difference of the third and first components of \eqref{equ:UV}, which is also the addition of $z-x$,
$$\big(z+(z-x)\big)-\big((x+(z-x)\big)=z-x,$$
so that $(UV)^2(x,y,z)$ is the addition of $(z-x)$ {\it twice} to each component. By induction, we have
\begin{equation} \label{equ:UVi}
(UV)^m(x,y,z)=\big(x,\;y,\;z\big)+m(z-x),
\end{equation}
so that $(UV)^{12}=\text{Id}$. The order of $UV$ is not smaller than 12, for instance $(UV)^m(1,3,2)$ is not $(1,3,2)$ for $m=1,\dots,11$ by \eqref{equ:UVi}, so $UV$ now has order 12.
\bigskip

To see that $UW$ also has order 12, we similarly observe that $UW(x,y,z)$ is the addition of $(z-y)$ to each component
\begin{equation} \label{equ:UW}
UW(x,y,z)=\big(x,\;y,\;z\big)+(z-y),
\end{equation}
and
\begin{equation} \label{equ:UWj}
(UW)^n(x,y,z)=\big(x,\;y,\;z\big)+n(z-y),
\end{equation}
and then argue as for $UV$.
\item
A straightforward computation shows $UVW(x,y,z)=(x,-y+2x,-z +2x)$ and $(UVW)^{2}=\text{Id}$.
\item \label{UVUW}
From equations \eqref{equ:UV} and \eqref{equ:UW} we see that both $$(UV)(UW)(x,y,z) \hspace{.25in} \text{and} \hspace{.25in} (UW)(UV)(x,y,z)$$ are the addition of $(z-x)+(z-y)$ to each component of $(x,y,z)$.
\item
Both $U$ and $V$ have order 2, so $U^{-1}(UV)U=VU=(UV)^{-1}$. Similarly, $U^{-1}(UW)U=WU=(UW)^{-1}$.
\end{enumerate}
\item
Since $U$, $V$, and $W$ each have order 2, any element of $\langle U, V, W \rangle$ is a word in the letters $U$, $V$, and $W$ in which no two consecutive letters are the same.

We first observe that length 1 words, i.e. the generators, can be put into the form \eqref{equ:normal_form}. Clearly $U$ can, and we use $V=U(UV)$ and $W=U(UW)$ for the other two generators.

We next describe how to transform a word of length 2 or more into the form \eqref{equ:normal_form} by converting two letters at a time (starting on the far right) into products of powers of $UV$ and $UW$. Consider a word in $U$, $V$, and $W$ in which no two consecutive letters are the same and in which there are 2 or more letters. If the 2 far right letters are $UV$ or $UW$, then we leave them as is. If they are $VU$ or $WU$ then we replace them by $(UV)^{11}$ or $(UW)^{11}$ respectively. If they are $VW$, then we rewrite $VW$ as
$$VW=(VU)(UW)=(UV)^{11}(UW).$$
Similarly, if they are $WV$, we rewrite as $(UW)^{11}(UV)$. Thus, in all of the possible cases, we have rewritten the two far right letters of the word as a product of powers of $UV$ and $UW$.

We similarly treat the third and fourth letters from the right, and so on, moving pairwise from right to left, until either no letters are left, or only one letter remains. If the remaining far left letter is $U$, then we are done. If the remaining far left letter is $V$, then we rewrite it as $U(UV)$. If the remaining far left letter is $W$, then we rewrite it as $U(UW)$.

We have now achieved $U^k$ followed by products of powers of $UV$ and $UW$. Finally we use the facts that $UV$ and $UW$ commute and have order 12 to move the $UV$'s left towards $U^k$ and the $UW$'s right, and bring the resulting word into the form \eqref{equ:normal_form}.

Next is uniqueness of the decomposition \eqref{equ:normal_form}. We claim $\langle UV \rangle \cap \langle UW \rangle=\{\text{Id}\}$. From equation \eqref{equ:UVi} we know $(UV)^m(x,y,z)$ is the addition of $m(z-x)$ in each component, and from \eqref{equ:UWj} we know $(UW)^n(x,y,z)$ is the addition of $n(z-y)$ in each component. To distinguish $(UV)^m$ and $(UW)^m$ we evaluate on $(1,2,3)$. To distinguish $(UV)^m$ and $(UW)^n$ for $m \neq n$ with $0 \leq m,n \leq 11$, we evaluate on $(1,1,2)$. Hence, $\langle UV \rangle \cap \langle UW \rangle=\{\text{Id}\}$, and as a consequence of the commutativity of $UV$ and $UW$, we see $\langle UV, UW\rangle$ is an internal direct product of $\langle UV \rangle$ and $\langle UW \rangle$, and isomorphic to $\mathbb{Z}_{12} \times \mathbb{Z}_{12}$.

We also claim $U \notin \langle UV, UW\rangle$. From the relations we already know, the only elements of order 2 in $\langle UV, UW\rangle$ are $(UV)^6$, $(UW)^6$, and $(UV)^6(UW)^6$. We can distinguish all these from $U$ on $(0,0,1)$ using \eqref{equ:UVi}, \eqref{equ:UWj}, and the proof of \ref{UVUW}.
$$U(0,0,1)=(0,0,-1) \hspace{.4in} (UV)^6(0,0,1)=(6,6,7) \hspace{.4in} (UW)^6(0,0,1)=(6,6,7)$$
$$(UV)^6(UW)^6(0,0,1)=(0,0,1)$$

For the uniqueness, suppose $U^k(UV)^m(UW)^n=U^p(UV)^q(UW)^r$ for some $k,p$ equal to 0 or 1 and some $m,n,q,r$ equal to $0, \dots,$ or $11$.
Then $U^{k-p}=(UV)^{q-m}(UW)^{r-n} \in \langle UV, UW\rangle$ and
$$U^{k-p}=\text{Id}=(UV)^{q-m}(UW)^{r-n},$$
so that $k=p$, $m=q$, and $n=r$.
\item
Immediately follows from \ref{thm:structure_of_J:semi-direct_product_form}.
\item
The 144-element group $\langle UV, UW\rangle$ has index 2 in $\calj$, so is normal. From the unique decomposition \eqref{equ:normal_form} we have $\langle U \rangle \cap \langle UV, UW\rangle = \{\text{Id}\}$ and $\calj= \langle U \rangle \langle UV, UW\rangle$ as sets. Finally, $\calj=\langle U \rangle \ltimes \langle UV, UW\rangle$ as groups. The conjugation action of $U$ on $\langle UV, UW\rangle$ is inversion by \eqref{U-conjugation_is_inversion}.
\item
The first equation follows from \eqref{equ:UVi} and \eqref{equ:UWj}. The second equation is an application of $U$ to the first equation, using linearity and $U(c,c,c)=(c,c,c)$.
\item
From the commutativity of $UV$ and $UW$ and \ref{item:action_of_normal_form}, we have
$$\aligned
(VW)^j(x,y,z)&=(VUUW)^j(x,y,z)\\
&=(UV)^{-j}(UW)^j(x,y,z)\\
&=\big(x,\;y,\;z\big)-j(z-x)+j(z-y)\\
&=\big(x,\;y,\;z\big)+j(x-z).
\endaligned$$
\end{enumerate}
\end{proof}

\begin{cor}
The group $\calj$ has a presentation of the form
\begin{equation} \label{equ:presentation_of_J}
\langle a,b,c \; \vert\; a^2,b^2,c^2,(abc)^2,(ab)^{12},(ac)^{12} \rangle.
\end{equation}
\end{cor}
\begin{proof}
By Theorem~\ref{thm:structure_of_J}~\ref{thm:structure_of_J:U_V_W_have_order_2}, \ref{thm:structure_of_J:UV_and_UW_have_order_12}, \ref{thm:structure_of_J:UVW_has_order_2}, the generators $U$, $V$, and $W$ of $\calj$ satisfy the relations indicated in \eqref{equ:presentation_of_J}. Also, $\calj$ has order 288.

Thus, since any group presented by \eqref{equ:presentation_of_J} surjects onto any other group satisfying the indicated relations (and perhaps more), it suffices to show that any group with presentation \eqref{equ:presentation_of_J} can have at most 288 elements. We do this by showing that the group \eqref{equ:presentation_of_J} also provides generators that satisfy the relations of the 288-element group $\mathbb{Z}_2 \ltimes (\mathbb{Z}_{12} \times \mathbb{Z}_{12})$ where $\mathbb{Z}_2$ acts by additive inversion.\footnote{Recall that a presentation of semi-direct product is given in terms of presentations of the constituent groups and the action, so we know the presentation of $\mathbb{Z}_2 \ltimes (\mathbb{Z}_{12} \times \mathbb{Z}_{12})$.} We suppose \eqref{equ:presentation_of_J} and claim that $a$, $ab$, and $ac$ satisfy the relations of $\mathbb{Z}_2 \ltimes (\mathbb{Z}_{12} \times \mathbb{Z}_{12})$. We already have $a^2=1$, and $(ab)^{12}=1=(ac)^{12}$. The commutativity of $ab$ and $ac$ follows from
$$1=(abc)(abc)=(ab)(ca)(ba)(ac)=(ab)(ac)^{-1}(ab)^{-1}(ac)$$
$$(ab)^{-1}(ac)^{-1}=(ac)^{-1}(ab)^{-1}$$
$$(ac)(ab)=(ab)(ac).$$
The semi-direct product action relation also follows from the assumption that $a$, $b$, and $c$ have order 2.
$$a^{-1}(ab)a=ba=(ab)^{-1}$$
$$a^{-1}(ac)a=ca=(ac)^{-1}$$
Thus we have surjective group homomorphisms
$$\xymatrix{\mathbb{Z}_2 \ltimes (\mathbb{Z}_{12} \times \mathbb{Z}_{12}) \ar[r] & \eqref{equ:presentation_of_J} \ar[r] & \calj}$$ where the first and last group have order 288. Therefore they must be isomorphisms.
\end{proof}

\begin{rmk}[$\calj \ncong \calu \ncong \langle PLR, T/I\rangle$]
Although both $\calj$ and $\calu$ have order 288 and are semi-direct products of $\mathbb{Z}_2$ and $\mathbb{Z}_{12}\times \mathbb{Z}_{12}$, they are not isomorphic, since the $\mathbb{Z}_2$ action is different. In $\calj$, the $\mathbb{Z}_2$ action is additive inversion on $\mathbb{Z}_{12}\times \mathbb{Z}_{12}$, while in $\calu$, the $\mathbb{Z}_2$ action exchanges the two copies of $\mathbb{Z}_{12}$. The group generated by the union of the $T/I$-group and the $PLR$-group also has order 288, but is isomorphic to neither $\calj$ nor $\calu$.
\end{rmk}

\begin{rmk}[Matrix Representation of Normal Form] \label{rem:matrix_representation_of_J_normal_form}
Evaluation of Theorem~\ref{thm:structure_of_J}~\ref{item:action_of_normal_form} on the standard basis yields the columns of the matrix representations of the elements of $\calj$ in the normal form of Theorem~\ref{thm:structure_of_J}~\ref{thm:structure_of_J:semi-direct_product_form}.
$$(UV)^m(UW)^n= \tiny  \left( \begin{array}{ccc} {1-m} & {-n} & {m+n} \\ {-m} & {1-n} & {m+n} \\ {-m} & {-n} &{1+m+n}\end{array} \right ), \;\; U(UV)^m(UW)^n=  \tiny  \left( \begin{array}{ccc} -m & 1-n & m+n \\ 1-m & -n & m+n \\ 1-m & 1-n  & -1+m+n \end{array} \right )$$
\end{rmk}

\begin{rmk}[Index of $\calj$ in $GL(3,\mathbb{Z}_{12})$]
The size of $GL(3,\mathbb{Z}_{12})$ is $$|GL(3,\mathbb{Z}_{12})|=|GL(3,\mathbb{Z}_{3})| \cdot |GL(3,\mathbb{Z}_{4})|.$$
The first factor has order
$$|GL(3,\mathbb{Z}_{3})|=(3^3-1)(3^3-3)(3^3-3^2)=2^5\cdot3^3\cdot13=11,232.$$
For the second factor, we use Theorem~1 of Hong--You \cite{HongYou} and take $p=2$, $\alpha=2$, $\beta=1$, and $n=3$ to obtain
$$\aligned
|GL(3,\mathbb{Z}_{4})|&=p^{n^2 \alpha}\left(1-\frac{1}{p^\beta} \right)\cdots \left(1-\frac{1}{p^{n\beta}} \right)\\
&=2^{3^2 \cdot 2}\left(1-\frac{1}{2} \right)\left(1-\frac{1}{2^2} \right)\left(1-\frac{1}{2^3} \right) \\
&=86,016.
\endaligned$$
Consequently, the 288-element group $\mathcal{J}$ has index $$3,354,624=\frac{11,232 \cdot 86,016}{288}$$
in $GL(3,\mathbb{Z}_{12})$.
\end{rmk}

\begin{rmk}[Index of $\calj$ in $SL(3,\mathbb{Z}_{12})$]
In general, the orbits of $SL(m,\mathbb{Z}_n)$ acting naturally on the Cartesian product $\mathbb{Z}_n^{\times m}$ were determined by Novotn\'{y}--Hrivn\'{a}k in \cite{NovotnyHrivnak}, and the order of $SL(m,\mathbb{Z}_n)$ was determined by Hong--You in \cite{HongYou}, see also \cite{FengOrdersOfClassicalGroups}. Using the formula of Hong--You, recalled in equation (2.3) of \cite{NovotnyHrivnak}, the order of the present group $SL(3,\mathbb{Z}_{12})$ is
$$241,532,928 = 12^8 \times \frac{3}{4} \times \frac{7}{8} \times \frac{8}{9} \times \frac{26}{27} =  12^8 \times  \left(1-\frac{1}{2^2}\right) \left(1-\frac{1}{2^3}\right) \left(1-\frac{1}{3^2}\right)\left(1-\frac{1}{3^3}\right) $$
Here we have taken $r = 2$ and $p_1 = 2$, $p_2 = 3$, and $12 = 2^2 \times 3^1$ in equation (2.3) of \cite{NovotnyHrivnak}.
Consequently, the 288-element group $\mathcal{J}$ has index 838,\,656 in $SL(3,\mathbb{Z}_{12})$.
\end{rmk}

\subsection{The Center of $\calj$ is a Klein 4-Group} \leavevmode \smallskip

We can now use Structure Theorem~\ref{thm:structure_of_J} to determine the center of $\calj$.

\begin{prop} \label{prop:Center_of_J}
The center of $\calj$ is the Klein 4-group $$\big\{\text{\rm Id},\; (UV)^6,\; (UW)^6,\; (UV)^6(UW)^6\big\}.$$
\end{prop}
\begin{proof}
Let $A \in \calj$ be in the abelian subgroup $\langle UV, UW \rangle \cong \mathbb{Z}_{12} \times \mathbb{Z}_{12}$. Then $A$ is in the center of $\calj$ if and only if $A$ commutes with $U$ (recall the normal form in Theorem~\ref{thm:structure_of_J}~\ref{thm:structure_of_J:semi-direct_product_form}). But $A$ commutes with $U$ if and only if $$A=U^{-1}AU=A^{-1},$$
where the last equality follows from Theorem~\ref{thm:structure_of_J}~\ref{U-conjugation_is_inversion}. Hence $A$ must have order 1 or 2 to be in the center. The only elements of order 1 or 2 in $\langle UV, UW \rangle$ are $\text{\rm Id}$, $(UV)^6$, $(UW)^6$, and $(UV)^6(UW)^6$.

Let $A \in \calj$ be in the abelian subgroup $\langle UV, UW \rangle$ (not necessarily in the center), and consider $UA$. We show that $UA$ cannot be in the center by contradiction. Suppose $UA$ commutes with $UV$. From Theorem~\ref{thm:structure_of_J}~\ref{U-conjugation_is_inversion} again we have
$$UV=(UA)(UV)(UA)^{-1}=U(A(UV)A^{-1})U^{-1}=U(UV)U^{-1}=U^{-1}(UV)U=(UV)^{-1}.$$
But $UV=(UV)^{-1}$ is impossible, as $UV$ has order 12. Hence $UA$ does not commute with $UV$, and $UA$ cannot be in the center.

We have now considered all elements of $\calj$ because of the normal form in Theorem~\ref{thm:structure_of_J}~\ref{thm:structure_of_J:semi-direct_product_form}.
\end{proof}

\subsection{The Centralizer of $\mathcal{J}$ in $GL(3, \mathbb{Z}_{12})$ is a Product of Klein 4-Groups} \label{subsec:JCentralizer_in_GL} \leavevmode \smallskip

In Proposition~\ref{prop:Center_of_J} we found the center of $\calj$ to be a Klein 4-group. Other elements of $GL(3, \mathbb{Z}_{12})$ that commute with $\calj$ are of course scalar multiplication with the units 1, 5, 7, 11 of $\mathbb{Z}_{12}$, that is, the four diagonal matrices with a single unit in all diagonal entries commute with $\calj$. These four matrices also form a Klein 4-group, as 5, 7, and 11 all have multiplicative order 2. The center of $\calj$ and these four matrices generate an internal direct product of two Klein 4-groups that commutes with $\calj$. We claim that the centralizer of $\calj$ in $GL(3, \mathbb{Z}_{12})$ consists of precisely these 16 matrices, and no more.

Further, we determine all not-necessarily-invertible matrices in $M(3, \mathbb{Z}_{12})$ that commute with $\calj$.

\begin{prop} \label{prop:Centralizer_of_J_in_GL3}
The group centralizer of $\calj$ in $GL(3, \mathbb{Z}_{12})$ consists of the following 16 matrices with $u=1,5,7,11.$
$$\text{\rm diag}(u)=\left( \begin{array}{ccc} u & 0 & 0 \\ 0 & u & 0 \\ 0 & 0 & u \end{array} \right )\hspace{.9in}\text{\rm diag}(u)\cdot(UV)^6(UW)^6=\left( \begin{array}{ccc} {u+6} & 6 & 0 \\ 6 & {u+6} & 0 \\ 6 & 6 & {u} \end{array} \right)
$$
$$\text{\rm diag}(u)\cdot(UV)^6=\left( \begin{array}{ccc} {u+6} & 0 & 6 \\ 6 & u & 6 \\ 6 & 0 & {u+6} \end{array} \right) \hspace{.25in}\text{\rm diag}(u)\cdot(UW)^6=\left( \begin{array}{ccc} u & 6 & 6 \\ 0 & {u+6} & 6 \\ 0 & 6 & {u+6} \end{array} \right) $$
This abelian group is an internal direct product of the Klein 4-group of the indicated diagonal matrices and the Klein 4-group of the center of $\calj$ from Proposition~\ref{prop:Center_of_J}.

The monoid centralizer of $\calj$ in the monoid of all matrices $M(3, \mathbb{Z}_{12})$ consists of the 30 matrices
with arbitrary $u \in \mathbb{Z}_{12}$ given by
$$\aligned
&\text{\rm diag}(u)    \hspace{1in}    &                 \text{\rm diag}(u)\cdot(UV)^6(UW)^6 \phantom{.} \\
&\text{\rm diag}(u)\cdot(UV)^6  \hspace{1in}     & \text{\rm diag}(u)\cdot(UW)^6.
\endaligned $$
When $u \in \mathbb{Z}_{12}$ is odd, these matrices take the form indicated above for $u$ invertible.
When $u \in \mathbb{Z}_{12}$ is even, these four matrices coincide and are all simply $\text{\rm diag}(u)$.

\end{prop}
\begin{proof}
Scalar multiplication and the center of $\calj$ clearly commute with all of $\calj$, so the indicated 16 elements are in the group centralizer of $\calj$, and the indicated 30 elements are in the monoid centralizer of $\calj$.

We first confirm that the indicated group elements have the claimed matrix forms. From  Remark~\ref{rem:matrix_representation_of_J_normal_form} we directly compute
\begin{equation} \label{equ:UV6}
(UV)^6=\left(\begin{array}{ccc} 7 & 0 & 6 \\ 6 & 1 & 6 \\ 6 & 0 & 7 \end{array} \right)
\end{equation}
and directly compute the matrices $(UW)^6$ and $(UV)^6(UW)^6$, and see that $(UW)^6$ and $(UV)^6(UW)^6$ are cyclic permutations of \eqref{equ:UV6} (i.e. both rows and columns are simultaneously cyclically permuted). Considering the effect of multiplying $u$ on entries of \eqref{equ:UV6}, if $u$ is odd, we have
$$u6=(2k+1)6=0+6=6$$
$$u7=u(1+6)=u+6,$$
and we obtain the claimed matrix form for $\text{\rm diag}(u)\cdot(UV)^6$ for $u$ odd. The claimed forms for $\text{\rm diag}(u)(UW)^6$ and $\text{\rm diag}(u)(UV)^6(UW)^6$ in the case $u$ odd follows similarly.

If $u$ is even, then
$$u6=0$$
$$u7=u(1+6)=u$$
and $\text{\rm diag}(u)\cdot(UV)^6=\text{\rm diag}(u)$, again using \eqref{equ:UV6}. Similarly,  $\text{\rm diag}(u)(UW)^6$ and $\text{\rm diag}(u)(UV)^6(UW)^6$ are just $\text{\rm diag}(u)$.

Next we show that no other matrices are in the monoid centralizer. We suppose $A\in M(3, \mathbb{Z}_{12})$ commutes with $\calj$, and then compute {\it only part} of the commutators with $U$, $V$, $W$, $UV$, $UW$, and $VW$. Since the commutators are zero, we have a family of equations which the entries of $A$ must satisfy, and these determine $A$. To avoid unnecessary computation, we only compute some of the commutator with $U$, and then only single entries of the other commutators that arise from a row/column {\it with a single $\pm 1$} in $V$, $W$, $UV$, $UW$, and $VW$. We use a $\ast$ to indicate entries we do not compute. Recall the matrix representations in Section~\ref{subsec:Definition_of_J} and Remark~\ref{rem:matrix_representation_of_J_normal_form}.
\begin{align}
AU-UA& =\left( \begin{array}{ccc} a & b & c \\ d & e & f \\ g & h & i \end{array} \right)\left( \begin{array}{ccc} 0 & 1 & 0 \\ 1 & 0 & 0 \\ 1 & 1 & -1 \end{array} \right)-\left( \begin{array}{ccc} 0 & 1 & 0 \\ 1 & 0 & 0 \\ 1 & 1 & -1 \end{array} \right) \left( \begin{array}{ccc} a & b & c \\ d & e & f \\ g & h & i \end{array} \right) \\
&= \left( \begin{array}{ccc} b+c-d & * & -c-f \\ e+f-a & * & * \\ h+i-(a+d-g) & * & * \end{array} \right)=\mathbf{0} \label{equ:centralizer_AU}
\end{align}

\begin{equation} \label{equ:centralizer_AV}
AV-VA=\left( \begin{array}{ccc} * & * & * \\ -d-g & * & * \\ * & * & * \end{array} \right)=\mathbf{0}
\end{equation}

\begin{equation} \label{equ:centralizer_AW}
AW-WA=\left( \begin{array}{ccc} * & -b-h & * \\ * & * & * \\ * & * & * \end{array} \right)=\mathbf{0}
\end{equation}

\begin{equation} \label{equ:centralizer_A(UV)}
A(UV)-(UV)A=\left( \begin{array}{ccc} * & b-h & * \\ * & * & * \\ * & * & * \end{array} \right)=\mathbf{0}
\end{equation}

\begin{equation}  \label{equ:centralizer_A(UW)}
A(UW)-(UW)A=\left( \begin{array}{ccc} * & * & * \\ d-g & * & * \\ * & * & * \end{array} \right)=\mathbf{0}
\end{equation}

\begin{equation}  \label{equ:centralizer_A(VW)}
A(VW)-(VW)A=\left( \begin{array}{ccc} * & * & * \\ * & * & f-c \\ * & * & * \end{array} \right)=\mathbf{0}
\end{equation}

A pairwise comparison of the two-variable equations above, namely \eqref{equ:centralizer_AU} with \eqref{equ:centralizer_A(VW)}, \eqref{equ:centralizer_AV} with \eqref{equ:centralizer_A(UW)}, and \eqref{equ:centralizer_AW} with \eqref{equ:centralizer_A(UV)}, reveals that all the non-diagonal entries of $A$, specifically $b$, $c$, $d$, $f$, $g$, $h$, must be 0 or 6, and
\begin{equation} \label{equ:off_diagonal_equalities}
c=f \hspace{.75in} d=g \hspace{.75in}  b=h.
\end{equation}
Consequently, in each column of $A$ the same value occurs in both off-diagonal positions.

The equation $b+c-d=0$ in the upper left of matrix \eqref{equ:centralizer_AU} then implies that either: $b$, $c$, and $d$ are all zero, {\it or} exactly two of $b$, $c$, and $d$ are 6 and the third is 0. We make this case distinction.

\begin{enumerate}
\item \label{item:two_zero_a}
Suppose $b$, $c$, and $d$ are all zero. \\
Then $f$, $g$, and $h$ are also zero, as  $f=c$, $g=d$, and $h=b$. The bottom two equations in the left column of matrix \eqref{equ:centralizer_AU} now imply $a=e=i$, so the diagonal entries of $A$ are equal, the non-diagonal entries are zero, and $A=\text{diag}(u)$ for some $u\in \mathbb{Z}_{12}$.
\item
Suppose exactly two of $b$, $c$, and $d$ are 6 and the third is 0. \\
We go through the three possibilities, and look at the lower two equations in the far left column of matrix \eqref{equ:centralizer_AU}. \\
\begin{enumerate}
\item \label{item:two_zero_a}
Suppose $b=c=6$, and $d=0$. \\
Then $f=h=6$ and $g=0$ by \eqref{equ:off_diagonal_equalities}. Then
$$0=e+f-a \Longrightarrow e=a+6,$$
$$0=\big(h+i-a-d+g\big)=\big(6+i-a-0+0\big) \Longrightarrow i=a+6,$$
and the lower two diagonal entries $e$ and $i$ are equal, while the first diagonal entry $a$ differs from them by 6.
\item \label{item:two_zero_b}
Suppose $c=d=6$ and $b=0$. \\
Then $f=g=6$ and $h=0$ by \eqref{equ:off_diagonal_equalities}.  Then
$$0=e+f-a \Longrightarrow e=a+6,$$
$$0=\big(h+i-a-d+g\big)=\big(0+i-a-6+6\big) \Longrightarrow i=a,$$
and the corner two diagonal entries $a$ and $i$ are equal, while the middle diagonal entry $e$ differs from them by 6.
\item \label{item:two_zero_c}
Suppose $d=b=6$ and $c=0$. \\
Then $g=h=6$ and $f=0$ by \eqref{equ:off_diagonal_equalities}.  Then
$$0=e+f-a \Longrightarrow e=a,$$
$$0=\big(h+i-a-d+g\big)=\big(6+i-a-6+6\big) \Longrightarrow i=a+6,$$
and the first two diagonal entries $a$ and $e$ are equal, while the lower corner diagonal entry $i$ differs from them by 6.
\end{enumerate}
Thus, in all three cases \ref{item:two_zero_a}, \ref{item:two_zero_b}, \ref{item:two_zero_c}, we see that the two diagonal entries in the columns with 6 coincide, and the third diagonal entry differs from these two diagonal entries by the residue 6.
\end{enumerate}

By this point, we have shown that any matrix $A \in M(3, \mathbb{Z}_{12})$ in the monoid centralizer has the form of the four matrix families indicated in the statement of the proposition, so we have determined the monoid centralizer.

Our final task is to determine the invertible elements in the monoid centralizer, i.e. to prove that such a matrix is invertible if and only if $u$ is a unit. We claim that the determinant of each matrix is $u^3$. The matrix $\text{diag}(u)$ clearly has determinant $u^3$. The determinant of the other 3 matrix families is
$$u\big((u+6)^2-36\big)=u\big(u^2+12u+36 - 36\big)=u^3.$$
Thus, in all cases $u^3$ must be a unit in order for the matrix to be invertible. Every unit in $\mathbb{Z}_{12}$ has multiplicative order 2, so $u^6=1$ and $u(u^5)=1$, so $u$ is a unit for all four matrix families, and we are finished.
\end{proof}

\subsection{The Centralizer of $\calj$ in $\text{Aff}(3,\mathbb{Z}_{12})$} \label{subsec:JCentralizer_in_Aff} \leavevmode \smallskip

In \cite[Theorem 3.2]{fiorenollsatyendraSchoenberg}, we proved that $U$, $V$, and $W$ commute\footnote{Actually, we proved this commutativity in the context of any $\mathbb{Z}_m$ rather than $\mathbb{Z}_{12}$. } with the component-wise application of any affine map $\mathbb{Z}_{12} \to \mathbb{Z}_{12}$. Next we can actually determine all affine endomorphisms of $\mathbb{Z}_{12}^{\times 3}$ that commute with $U$, $V$, and $W$.

\begin{prop} \label{prop:Centralizer_of_J_in_Aff(3,Z12)}
Let $\text{\rm Aff}(3,\mathbb{Z}_{12})$ denote the monoid of affine endomorphisms of $\mathbb{Z}_{12}^{\times 3}$ and let
$\text{\rm Aff}^\times (3,\mathbb{Z}_{12})$ denote its group of invertible elements. The monoid centralizer of $\calj$ in $\text{\rm Aff}(3,\mathbb{Z}_{12})$ consists of those maps $x \mapsto Ax+(q,q,q)$ where $A$ is any of the 30 matrices in $M(3,\mathbb{Z}_{12})$ that commute with $\calj$ as determined in Proposition~\ref{prop:Centralizer_of_J_in_GL3}, and $q \in \mathbb{Z}_{12}$. In particular, the $T/I$-group commutes with $\calj$.

The centralizer of $\mathcal{J}$ in $\text{\rm Aff}^\times (3,\mathbb{Z}_{12})$ is isomorphic to the semi-direct product of its centralizer in $GL(3,\mathbb{Z}_{12})$ with $\mathbb{Z}_{12}$, where we consider $\mathbb{Z}_{12}$ as embedded into $\mathbb{Z}_{12}^{\times 3}$ via the ``diagonal'' embedding $q \mapsto (q,q,q)$. The centralizer in $GL(3,\mathbb{Z}_{12})$ was determined in Proposition~\ref{prop:Centralizer_of_J_in_GL3}.
\end{prop}
\begin{proof}
We use homogeneous coordinates to notate affine maps.

Suppose an affine endomorphism $x \mapsto Ax+b$ commutes with all $J \in \calj$. Then
$$\aligned
\left( \begin{array}{c|c} 0 & 0 \\ \hline 0 & 1\end{array} \right ) &=
\left( \begin{array}{c|c}  A & b \\ \hline 0 & 1\end{array} \right ) \cdot \left( \begin{array}{c|c}  J & 0 \\ \hline 0 & 1\end{array} \right ) -  \left( \begin{array}{c|c}  J & 0 \\ \hline 0 & 1\end{array} \right ) \cdot \left( \begin{array}{c|c}  A & b \\ \hline 0 & 1\end{array} \right ) \\
&=   \left( \begin{array}{c|c} A  J - J A & b - J b \\ \hline 0 & 1\end{array} \right )
\endaligned$$
and $AJ-JA=\mathbf{0}$ for all $J \in \calj$, so $A$ is in the monoid centralizer of $\calj$. We also have $b - Jb=\mathbf{0}$ for all $J \in \calj$, in particular $b$ is fixed by the generators $U$, $V$, and $W$. Already for $U$ and $V$, we see from
$$U \left( \begin{array}{c}b_1 \\ b_2 \\ b_3 \end{array} \right ) = \left( \begin{array}{c}b_2 \\ b_1 \\ b_1 + b_2 - b_3 \end{array} \right ) \hspace{.4in} \text{and} \hspace{.4in}
V \left( \begin{array}{c}b_1 \\ b_2 \\ b_3 \end{array} \right ) = \left( \begin{array}{c}- b_1 + b_2 + b_3 \\ b_3 \\ b_2 \end{array} \right )$$
that $b_1=b_2=b_3$.

Thus, if $x \mapsto Ax+b$ commutes with all $J \in \calj$, then $A$ is in the monoid centralizer of $\calj$ and the translation vector $b$ has the same entry in all three components. The converse is clearly also true.

Since $\text{\rm Aff}^\times (3,\mathbb{Z}_{12}) = GL(3,\mathbb{Z}_{12}) \rtimes \mathbb{Z}_{12}^{\times 3}$, the structure claim follows.
\end{proof}

\subsection{Recollection on Permutation Matrices} \label{subsec:Recollection_of_Permutations} \leavevmode \smallskip

We denote by $\Sigma_3$ the permutation group on the set $\{1,2,3\}$, and we follow the standard function orthography in which the rightmost function is done first. The notation (123) is cycle notation for the permutation $1 \mapsto 2 \mapsto 3 \mapsto 1$. For any set $X$, the left action of $\Sigma_3$ on the Cartesian product $X^3$ is
$$\sigma (x_1,x_2,x_3)=(x_{\sigma^{-1} 1},x_{\sigma^{-1} 2},x_{\sigma^{-1} 3}).$$
When $X=\mathbb{Z}_{12}$, this left action of $\Sigma_3$ on $\mathbb{Z}_{12} \times \mathbb{Z}_{12} \times \mathbb{Z}_{12}$ arises from the left action of $\Sigma_3$ on the standard column vector basis $\{e_1, e_2, e_3\}$ via
$$\sigma e_i = e_{\sigma i}.$$
In this way, the $3 \times 3$ matrix $P_\sigma$ corresponding to $\sigma$ has columns $e_{\sigma 1}\;\; e_{\sigma 2}\;\; e_{\sigma 3}$. See for instance \cite[Section 5.1, Exercises 7,8,9]{DummitFoote}.

For example, for $\sigma=(1\;2\;3)$, we have $\sigma (x_1,x_2,x_3)=(x_3,x_1,x_2)$ and
$$\left( \begin{array}{ccc} 0 & 0 & 1 \\ 1 & 0 & 0 \\ 0 & 1 & 0 \end{array} \right )\left( \begin{array}{c} x_1 \\ x_2 \\ x_3 \end{array} \right ) = \left( \begin{array}{c} x_3 \\ x_1 \\ x_2 \end{array} \right ).$$

For readability, we always just write $\sigma$ for $P_\sigma$ when the context is clear. No confusion between cycles $(1\;2\;3)$ and vectors $(1,2,3)$ can arise, because cycles have no commas, while vectors have commas.

Permutations were incorporated into neo-Riemannian duality in \cite{fiorenollsatyendraMCM2013}.

\subsection{The Group $\langle \Sigma_3, \mathcal{J} \rangle$ is $\Sigma_3 \ltimes \mathcal{J}$} \label{subsec:Sigma3_J} \leavevmode \smallskip

Conjugation by elements of $\sigma_3$ permutes the generators $U$, $V$, and $W$.

\begin{prop} \label{prop:Sigma3_acts_on_J}
Recall the standard left action on 3-tuples from Section~\ref{subsec:Recollection_of_Permutations}, and recall the notation
$J^{1,2}$, $J^{2,3}$, and $J^{3,1}$ for $U$, $V$, and $W$ in Section~\ref{subsec:Definition_of_J}. \\
For $\sigma \in \Sigma_3$ we have the following compatibilities.
\begin{enumerate}
\item \label{prop:Sigma3_acts_on_J:i}
$\sigma J^{r,s} = J^{\sigma r, \sigma s} \sigma$
\item \label{prop:Sigma3_acts_on_J:ii}
$\sigma J^{r,s} \sigma^{-1}= J^{\sigma r, \sigma s}$
\item \label{prop:Sigma3_acts_on_J:iii}
$J^{r,s} \sigma=\sigma J^{\sigma^{-1}r, \sigma^{-1}s}$.
\end{enumerate}
\end{prop}
\begin{proof}
For \ref{prop:Sigma3_acts_on_J:i}, let $i,j \in \{1,2,3\}$ be such that $r=\sigma^{-1} i$ and $s=\sigma^{-1} j$. We follow Convention~\ref{conv:adding_constant_to_3tuple}, and have indicated where the definition of the $J$-operators from Section~\ref{subsec:Definition_of_J} are used.
$$\aligned
\sigma J^{r,s}(x_1,x_2,x_3) & \overset{\text{def}}{=} \sigma I_{x_r+x_s}(x_1,x_2,x_3) \\
 & \overset{}{=} \sigma \Big(-\big(x_1,\; x_2,\; x_3\big) + (x_r+x_s) \Big) \\
 & \overset{}{=}  -\Big(x_{\sigma^{-1}1},\; x_{\sigma^{-1}2},\; x_{\sigma^{-1}3} \Big) + (x_r+x_s)  \\
 & \overset{}{=}  -\Big(x_{\sigma^{-1}1},\; x_{\sigma^{-1}2},\; x_{\sigma^{-1}3} \Big) + (x_{\sigma^{-1} i}+x_{\sigma^{-1} j})  \\
 & \overset{}{=} I_{x_{\sigma^{-1}i}+x_{\sigma^{-1}j}}(x_{\sigma^{-1}1}, x_{\sigma^{-1}2},x_{\sigma^{-1}3})\\
 & \overset{\text{def}}{=} J^{i,j}(x_{\sigma^{-1}1}, x_{\sigma^{-1}2},x_{\sigma^{-1}3}) \\
 & \overset{}{=} J^{\sigma r,\sigma j}\sigma(x_1,x_2,x_3)
\endaligned$$
Claim \ref{prop:Sigma3_acts_on_J:ii} follows directly from \ref{prop:Sigma3_acts_on_J:i} by right multiplication with $\sigma^{-1}$, while \ref{prop:Sigma3_acts_on_J:iii} follows from \ref{prop:Sigma3_acts_on_J:i} by replacing $\sigma$ by $\sigma^{-1}$ and multiplying.
\end{proof}

\begin{prop} \label{prop:Sigma3_J_is_SemiDirectProduct}
The subgroup $\langle \Sigma_3, \mathcal{J} \rangle$ of $GL(3,\mathbb{Z}_{12})$ generated by the permutation matrices and the group $\mathcal{J}$ is the semi-direct product $\Sigma_3\ltimes \mathcal{J}$.
\end{prop}
\begin{proof}
Recall that $U$, $V$, and $W$ are involutions, so every element of $\mathcal{J}=\langle U,V,W\rangle$ can be written as a concatenation of $U$, $V$, and $W$.

The group $\calj$ is normal in $\langle \Sigma_3, \mathcal{J} \rangle$ because if we have such a concatenation $j_1 j_2 \cdots j_n$, and $\sigma \in \Sigma_3$, then
$$\sigma j_1 j_2 \cdots j_n \sigma^{-1} = \left( \sigma j_1 \sigma^{-1}\right) \left( \sigma j_2 \sigma^{-1} \right) \cdots \left( \sigma j_n \sigma^{-1} \right) \in \mathcal{J}$$
by Proposition~\ref{prop:Sigma3_acts_on_J}~\ref{prop:Sigma3_acts_on_J:ii}.

We have $\langle \Sigma_3, \mathcal{J} \rangle=\Sigma_3\mathcal{J}$ because if we have any concatenation of $U$, $V$, $W$ and elements of $\Sigma_3$, we can move all the permutation matrices to the left via Proposition~\ref{prop:Sigma3_acts_on_J}~\ref{prop:Sigma3_acts_on_J:iii}, and obtain a concatenation of $U$, $V$, $W$ on the right, in total an element of $\Sigma_3\mathcal{J}$.

Finally, we also have $\Sigma_3 \cap \calj = \{\text{Id}\}$ because we know that the restriction of $\calj$ to its six orbits is the $PLR$-group (in various correspondences), and no element of the $PLR$-group acts as a permutation of vector entries.
\end{proof}

\begin{prop}
The centralizer of $\calj$ in $GL(3, \mathbb{Z}_{12})$ is stable under conjugation by permutations as a set. If $C$ commutes with $U$, $V$, and $W$, then so does $\sigma C \sigma^{-1}$ for all $\sigma \in \Sigma_3$.
\end{prop}
\begin{proof}
Suppose $C$ commutes with $U$, $V$, and $W$, and let $\sigma \in \Sigma_3$. Then
$$\aligned
J^{r,s}C &= CJ^{r,s} \\
\sigma (J^{r,s}C) \sigma^{-1} &= \sigma (C J^{r,s}) \sigma^{-1} \\
J^{\sigma r, \sigma s} (\sigma C \sigma^{-1}) &= (\sigma  C \sigma^{-1}) J^{\sigma r,\sigma s}
\endaligned$$
where we use Proposition~\ref{prop:Sigma3_acts_on_J}~\ref{prop:Sigma3_acts_on_J:i} and \ref{prop:Sigma3_acts_on_J:iii} in the last step.
\end{proof}

\subsection{Subgroups of $\Sigma_3 \ltimes \calj$ and their Triadic Orbits} \label{subsec:Triadic_Orbits} \leavevmode \smallskip

We now consider subgroups of $\Sigma_3 \ltimes \calj$ and their relevant triadic orbits. These orbits will be of use in Section~\ref{subsec:Hook_Group_And_Its_Structure} when we study the Hook group $\calh$ and its properties. The present section also elucidates $\calj$ and $\Sigma_3 \ltimes \calj$ as triadic transformation groups.

The Structure Theorem~\ref{thm:structure_of_J} provides a good understanding of $\calj$. The group $\calj$ contains an index 2 commutative subgroup
$$\calj^+:=\big\{(UV)^m(UW)^n \;\vert\; m, n = 0, 1, \dots, 11\big\}$$
of operations that preserve mode, that is, which map major triads to major triads, and minor triads to minor triads, no matter the voicing (recall from Theorem~\ref{thm:structure_of_J}~\ref{item:action_of_normal_form} that $(UV)^m(UW)^n$ adds a constant in each component).
The other coset
$$\calj^-:=U\calj^+=\left\{U(UV)^m(UW)^n \;\vert\; m, n = 0, 1, \dots, 11\right\}$$
consists of operations that reverse mode, that is, send major triads to minor triads, and minor triads to major triads, no matter the voicing.

Clearly, permutations also preserve the modes of consonant triads. So we extend the foregoing discussion to include permutations as well. The group of {\it mode-preserving operations} in $\Sigma_3 \ltimes \calj$ is
$$\Sigma_3 \ltimes \calj^+:=\Sigma_3 \calj^+ = \big\{\sigma(UV)^m(UW)^n \;\vert\; \sigma \in \Sigma_3, \; m, n = 0, 1, \dots, 11\big\},$$
and the coset of {\it mode-reversing operations} in $\Sigma_3 \ltimes \calj$ is
$$\Sigma_3 \calj^- = \big\{ \sigma U(UV)^m(UW)^n \;\vert\; \sigma \in \Sigma_3, \; m, n = 0, 1, \dots, 11\big\}.$$
In summary we have
$$\calj=\calj^+ \bigsqcup \; \calj^- \hspace{.75in} \text{and} \hspace{.75in} \Sigma_3\calj=\Sigma_3\calj^+ \bigsqcup \; \Sigma_3\calj^-.$$

We define $\calh$ to be the subgroup of $\Sigma_3 \ltimes \calj$ that maps root position triads to root position triads, and we call $\calh$ the {\it Hook group}. We will study its structure in Section~\ref{subsec:Hook_Group_And_Its_Structure} and see that it is exactly the image of the representation $\rho$ from Section~\ref{subsec:defn_of_rho}. For the moment, we introduce the following collections of triads, and observe that they are orbits of the $C$-major or $C$-minor triad under appropriate groups.

$$\aligned
\texttt{Triads}&:=\text{set of all $144$ consonant triads in any of the 6 voicings} \\
&= \left\{ \left( \begin{array}{c} 0 \\ 4 \\ 7 \end{array} \right),
\left(\begin{array}{c} 4 \\ 7 \\ 0 \end{array}\right),
\left(\begin{array}{c} 7 \\ 0 \\ 4  \end{array}\right),
\left(\begin{array}{c} 0 \\ 7 \\ 4 \end{array}\right),
\left(\begin{array}{c} 4 \\ 0 \\ 7  \end{array}\right),
\left(\begin{array}{c} 7 \\ 4 \\ 0  \end{array}\right), \right.\\
&\phantom{=} \hspace{.2in} \left. \left( \begin{array}{c} 0 \\ 3 \\ 7 \end{array} \right),
\left(\begin{array}{c} 3 \\ 7 \\ 0 \end{array}\right),
\left(\begin{array}{c} 7 \\ 0 \\ 3  \end{array}\right),
\left(\begin{array}{c} 0 \\ 7 \\ 3 \end{array}\right),
\left(\begin{array}{c} 3 \\ 0 \\ 7  \end{array}\right),
\left(\begin{array}{c} 7 \\ 3 \\ 0  \end{array}\right), \text{ et cetera} \right\} \\
&=\text{orbit of $(0,4,7)$ under action of $\Sigma_3 \ltimes \calj$ }
\endaligned$$

$$\aligned
\texttt{MajTriads}&:=\text{set of all 72 major triads in any of the 6 voicings} \\
&= \left\{ \left( \begin{array}{c} 0 \\ 4 \\ 7 \end{array} \right),
\left(\begin{array}{c} 4 \\ 7 \\ 0 \end{array}\right),
\left(\begin{array}{c} 7 \\ 0 \\ 4  \end{array}\right),
\left(\begin{array}{c} 0 \\ 7 \\ 4 \end{array}\right),
\left(\begin{array}{c} 4 \\ 0 \\ 7  \end{array}\right),
\left(\begin{array}{c} 7 \\ 4 \\ 0  \end{array}\right),  \text{ et cetera} \right\} \\
&=\text{orbit of $(0,4,7)$ under action of $\Sigma_3 \ltimes \calj^+$ }
\endaligned$$

$$\aligned
\texttt{MinTriads}&:=\text{set of all 72 minor triads in any of the 6 voicings} \\
&= \left\{ \left( \begin{array}{c} 0 \\ 3 \\ 7 \end{array} \right),
\left(\begin{array}{c} 3 \\ 7 \\ 0 \end{array}\right),
\left(\begin{array}{c} 7 \\ 0 \\ 3  \end{array}\right),
\left(\begin{array}{c} 0 \\ 7 \\ 3 \end{array}\right),
\left(\begin{array}{c} 3 \\ 0 \\ 7  \end{array}\right),
\left(\begin{array}{c} 7 \\ 3 \\ 0  \end{array}\right),  \text{ et cetera} \right\} \\
&=\text{orbit of $(0,3,7)$ under action of $\Sigma_3 \ltimes \calj^+$ }
\endaligned$$

$$\aligned
\texttt{RootPosTriads}&:=\text{set of all 24 consonant triads in root position} \\
&= \left\{ \left( \begin{array}{c} 0 \\ 4 \\ 7 \end{array} \right),
\left(\begin{array}{c} 1 \\ 5 \\ 8 \end{array}\right),
\left(\begin{array}{c} 2 \\ 6 \\ 9  \end{array}\right),
\dots,
\left(\begin{array}{c} 0 \\ 3 \\ 7 \end{array}\right),
\left(\begin{array}{c} 1 \\ 4 \\ 8  \end{array}\right),
\left(\begin{array}{c} 2 \\ 5 \\ 9  \end{array}\right),  \dots \right\} \\
&=\text{orbit of $(0,4,7)$ under action of $\calh$}
\endaligned$$

$$\aligned
\texttt{DualRootPosTriads}&:=\text{set of all 24 consonant triads in dualistic root position} \\
&= \left\{ \left( \begin{array}{c} 0 \\ 4 \\ 7 \end{array} \right),
\left(\begin{array}{c} 1 \\ 5 \\ 8 \end{array}\right),
\left(\begin{array}{c} 2 \\ 6 \\ 9  \end{array}\right),
\dots,
\left(\begin{array}{c} 7 \\ 3 \\ 0 \end{array}\right),
\left(\begin{array}{c} 8 \\ 4 \\ 1  \end{array}\right),
\left(\begin{array}{c} 9 \\ 5 \\ 2  \end{array}\right),  \dots \right\} \\
&=\text{orbit of $(0,4,7)$ under action of $\calj$ }
\endaligned$$

The penultimate claim that \texttt{RootPosTriads} is the $\calh$-orbit of $(0,4,7)$ follows from equations \eqref{equ:rho_does_as_should_on_+} and \eqref{equ:rho_does_as_should_on_-} and the that fact that $\calh$ is the image of $\rho$ in Section~\ref{subsec:Hook_Group_and_UTTs}.

\subsection{Computer Observations about $\Sigma_3 \ltimes \calj$} \label{subsec:Computer_Observations_On_Traces} \leavevmode \smallskip

$$\text{trace}\Big((UV)^m(UW)^n\Big)=3 \hspace{1in} \text{trace}\Big(U(UV)^m(UW)^n\Big)=-1$$
For any cyclic permutation $(a\;b\;c) \in \Sigma_3$, we have
$$\text{trace}\Big((a\;b\;c)(UV)^m(UW)^n\Big)=0 \hspace{1in} \text{trace}\Big((a\;b\;c) U(UV)^m(UW)^n\Big)=2$$
For any transposition $(a\;b) \in \Sigma_3$, we have
$$\text{trace}\Big((a\;b)(UV)^m(UW)^n\Big)=1 \hspace{1in} \text{trace}\Big((a\;b) U(UV)^m(UW)^n\Big)=1$$

The conjugation class of $U$ in $\calj$ consists of 36 elements, while the conjugation class of $U$ in $\Sigma_3 \ltimes \calj$ has 108 elements.

\section{Musical Examples and Music-Theoretical Consequences} \label{sec:Musical_Examples_And_Consequences}

\subsection{Hexatonic Grail Motive as a Cycle} \label{subsec:Grail_Motive} \leavevmode \smallskip

The Grail motive in Wagner's {\it Parsifal}, Act 3, measures 1098 -- 1100 is harmonized by the consonant triads $E\flat$, $b$, $G$, $e\flat$ which eventually lead via $A\flat$ to $D\flat$. Following David Clampitt's \cite{clampittParsifal} hexatonic reading of the first four chords we extrapolate from the upper three voices a full hexatonic cycle leading back to $E\flat$ as shown in the upper part of Figure \ref{fig:ParsifalScore_HexatonicCycle}.
\begin{figure}[h]
  \centering
  \includegraphics[width=4in]{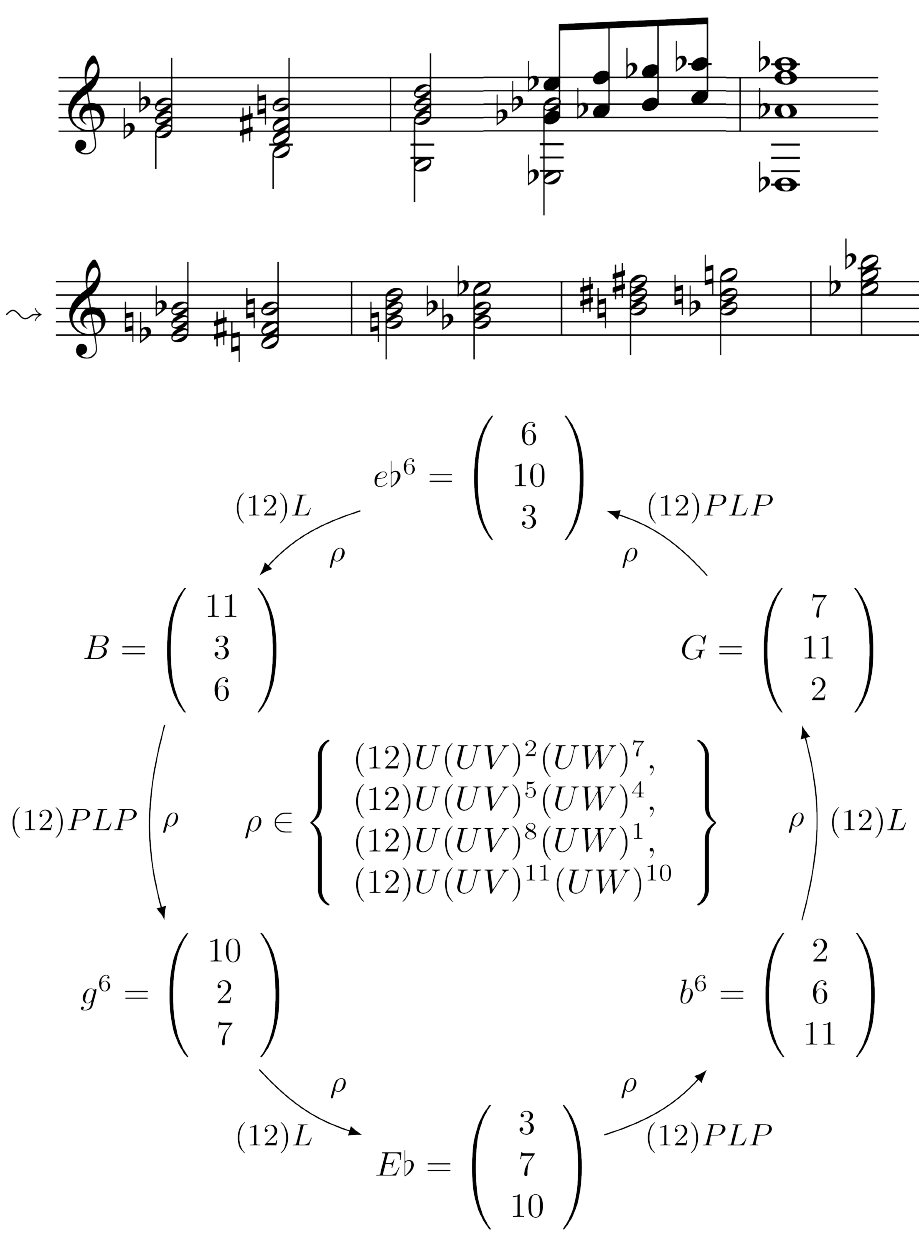}\\
  \caption{The two musical staves show the extrapolation of a hexatonic cycle of alternating root position major triads and first inversion minor traids from the first four chords in the harmonization of Grail motive in Wagner's {\it Parsifal}, Act 3, measures 1098--1100. The top staff depicts David Clampitt's \cite{clampittParsifal} reduction of these measures. The staff below depicts our hexatonic extrapolation. The hexagonal network provides a transformational analysis of this extrapolation. The outer labels $(12)PLP$ and $(12)L$ denote two contextual transformations between the chosen voicings (i.e. root position major triads and first inversion minor triads). The inner label $\rho$ stands for any of the four elements of $\Sigma_3 \ltimes \calj$ that provide a cyclic orbit along these six voicings. The normal forms of these group elements are depicted in the center of the network.} \label{fig:ParsifalScore_HexatonicCycle}
\end{figure}

The lower part of this figure provides a transformational network for these six voicings with two readings. The outer labels $(12)PLP$ and $(12)L$ denote two contextual transformations which form a hexatonic flip-flop-cycle. The inner label $\rho$ stands for four elements of $\Sigma_3 \ltimes \calj$, each of which yields a cyclic orbit along these six voicings.
To find the linear transformations with this cyclic orbit, we recall Theorem~\ref{thm:structure_of_J}~\ref{item:action_of_normal_form} and use the first three sequence elements to produce a system of 2 equations in 2 unkowns. {\small $$U(3,7,10)+m(10-3)+n(10-7)=(6,2,11) \hspace{.2in} U(2,6,11)+m(11-2)+n(11-6)=(11,7,2)$$}
\noindent Notice that the output chords are in the $T/I$-class of the input, but need to be reordered by permutation (12) to match the sequence. The two equations reduce to $4m=8$ and $n=1+3m$, so solutions are \;
\begin{tabular}{|c|c|c|c|c|}
\hline
$m$ & 2 & 5 & 8 & 11 \\ \hline
$n$ & 7 & 4 & 1 & 10 \\ \hline
\end{tabular}\; , and the elements of $\Sigma_3 \ltimes \calj$ with orbit the Grail sequence are
{\small
$$\begin{array}{llllll} (12) U (UV)^2 (UW)^7 & = & \left (  \begin{array}{ccc} {11} & 5 & 9 \\ {10} & 6 & 9 \\ {11} & 6 & 8 \end{array} \right ), \, \, &
(12) U (UV)^5 (UW)^4 & = & \left (  \begin{array}{rrr} 8 & 8 & 9 \\ 7 & 9 & 9 \\ {8} & 9 & 8 \end{array} \right ), \\ &&&&&\\
(12) U (UV)^8 (UW)^1 & = & \left (  \begin{array}{rrr} 5 & 11 & 9 \\ 4 & 0 & 9 \\ {5} & 0 & 8 \end{array} \right ), \, \, &
 (12) U (UV)^{11} (UW)^{10} & = & \left (  \begin{array}{rrr} 2 & 2 & 9 \\ 1 & 3 & 9 \\ {2} & 3 & 8 \end{array} \right ).
\end{array}$$
}

\subsection{Recalcitrant Viola in Schoenberg, String Quartet in $D$ minor, op. 7} \label{subsec:Recalcitrant Viola} \leavevmode \smallskip

We revisit an analysis of a triadic sequence in Schoenberg, String Quartet in $D$ minor, op.~7 by Fiore--Noll--Satyendra  \cite{fiorenollsatyendraSchoenberg}. Our group $\calj$ affords us a more economical description, allowing us to replace both $(13)P$ and $(13)R$ by $(13)V$. We also complement the work of \cite{fiorenollsatyendraSchoenberg} to include two final 3-pitch-class sequences in the viola using $(13)L$, which we promptly also replace by $(13)V$. We use the term {\it segment} to refer to 3-note sequences such as $(1,6,10)$, {\it et cetera}.

\begin{figure}[h]
    \centering
    \begin{subfigure}[b]{1\textwidth}
        \centering
        \includegraphics[width=4.8in]{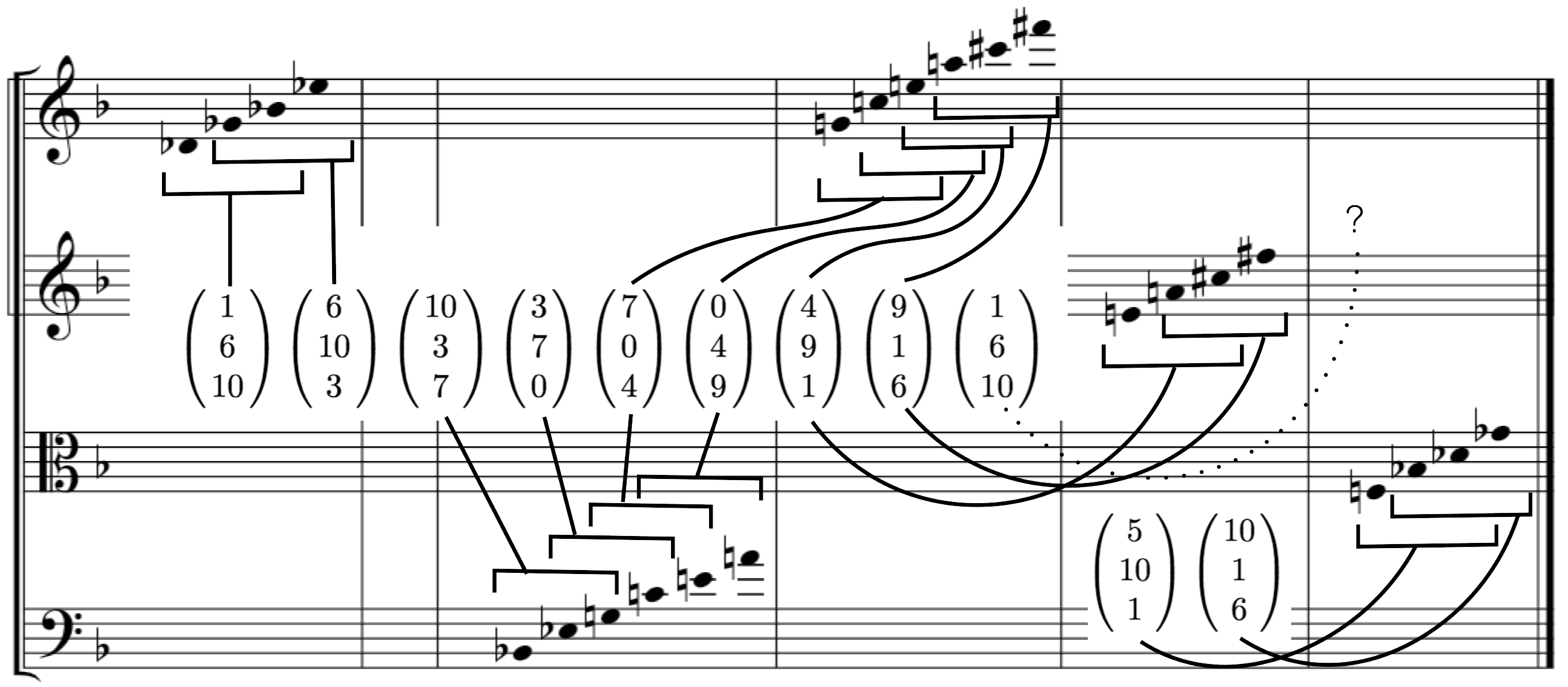}
        \caption{Schoenberg, String Quartet in $D$ minor, op. 7, measures 88--93. See Figure~\ref{fig:Schoenberg_Cycles} for a mapping of these measures in actual and implied cycles.\label{fig:Schoenberg_Partitur} }
    \end{subfigure}
    \\
    \begin{subfigure}[b]{1\textwidth}
        \centering
        \includegraphics[width=4.46in]{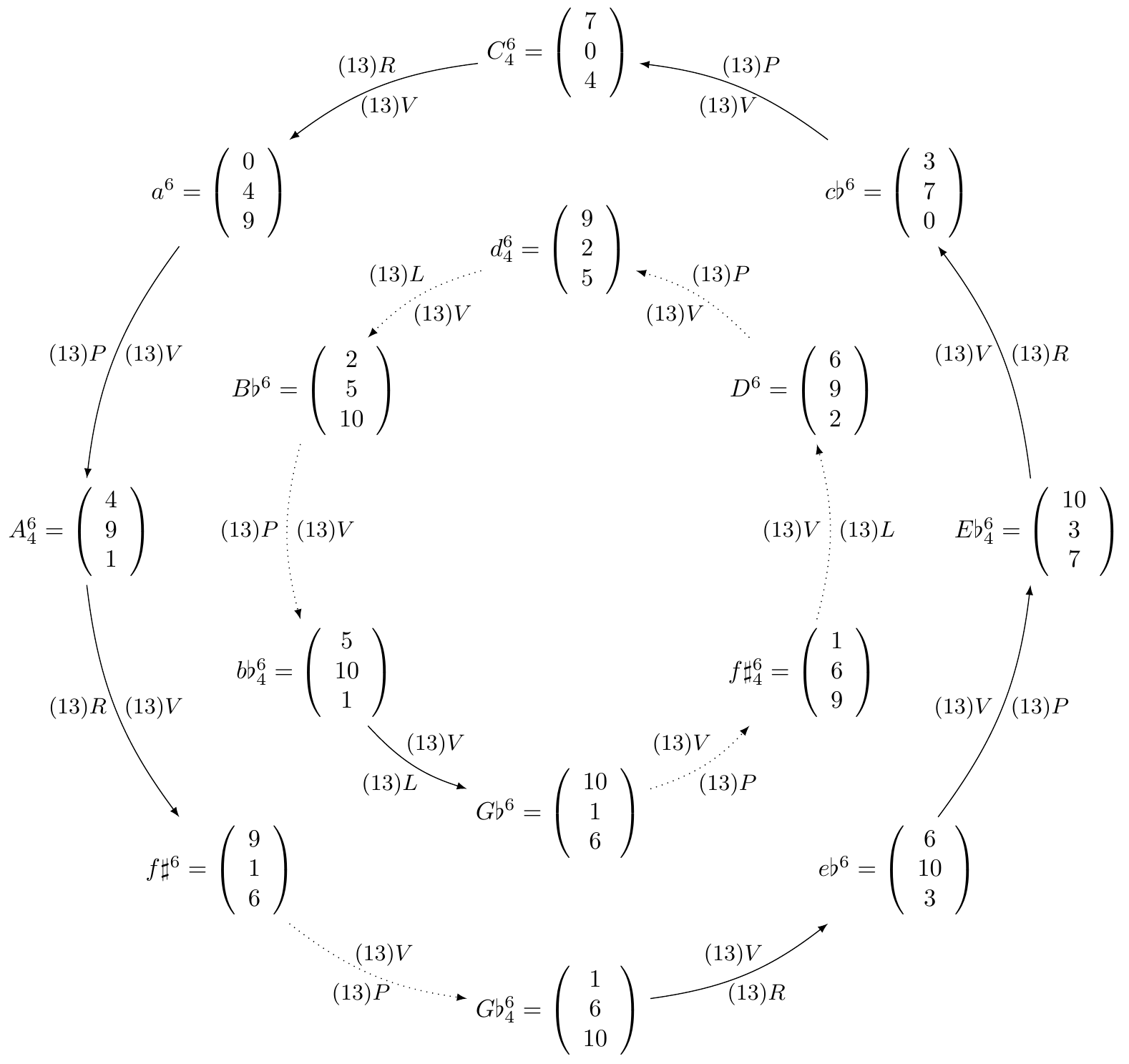}
        \caption{Cyclic networks extrapolating connections between triadic segments in Schoenberg, String Quartet in $D$ minor, op. 7, measures 88--93. The outer cycle is the octatonic $PR$-cycle of the first 9 triadic segments of Figure~\ref{fig:Schoenberg_Partitur}, while the inner cycle is the hexatonic $PL$-cycle {\it implied} by the final viola $L$-relationship $b\flat^6_4 \mapsto G\flat^6$ and the $P$-relationship between the final $G\flat$ and the earlier $f\sharp$ in the 2nd violin. All implied edges are indicated by dotted arrows. \label{fig:Schoenberg_Cycles}}
    \end{subfigure}
   \caption{}
\end{figure}

The first 9 consonant triadic segments of Figure~\ref{fig:Schoenberg_Partitur} form a complete enchained octatonic $PR$-cycle (the transformations $P$ and $R$ are defined via local conjugation of standard dualistic $P$ and $R$ by permutations, see Section~\ref{sec:Recollection_of_PLR}). The $G\flat$-segment in ordering $(1,6,10)$ in the ninth position does not actually occur in the score, as we indicate with a dotted line and question mark. Instead, the $G\flat$-segment appears in ordering $(10,1,6)$ as the final viola notes. The penultimate segment $b\flat$ in voicing $(5,10,1)$, after the $PR$-cycle, does not belong to the octatonic $PR$-cycle, but instead stands in an $L$-relationship to the final $G\flat$-segment, which in turn stands in a $P$-relationship to the preceding $f\sharp$.

In Figure~\ref{fig:Schoenberg_Cycles} we map the triadic segments as an enchained octatonic $PR$-cycle and an implied enchained hexatonic $PL$-cycle. The final $G\flat$ pitch-class set $\{6,10,1\}$, common to both cycles, is at the bottom of both cycles in its two relevant voicings. In Figure~\ref{fig:Schoenberg_Cycles}, all the arrows have two labels: $(13)V$ and one of $(13)P$, $(13)L$, $(13)R$. The two labels illustrate how the single transformation $(13)V$ offers a more economical description than the other three together. The transformation $(13)V$ is equal to RICH, {\it retrograde inversion enchaining}. All consonant cycles for RICH were determined in Table~1 on page 113 of \cite{fiorenollsatyendraMCM2013}. See Straus \cite{StrausContextualInversions} for some analyses involving RICH, one of which we revisited in the motivational {\it Problembeispiel} in Section~\ref{subsec:MotivationalProblembeispiel}.

\subsection{Affine Morphisms of Generalized Interval Systems} \label{subsec:Affine_Morphisms} \leavevmode \smallskip

Our determination of the centralizer of $\calj$ in the monoid $\text{Aff}(3,\mathbb{Z}_{12})$ in Proposition~\ref{prop:Centralizer_of_J_in_Aff(3,Z12)} is relevant for constructing morphisms of generalized interval systems, as we very briefly indicate with a few examples.

Consider again the first 9 consonant triads indicated in Figure~\ref{fig:Schoenberg_Partitur} and in the outer ring of Figure~\ref{fig:Schoenberg_Cycles}. The image of these interlocking consonant triads under the affine endomorphism $x \mapsto 7x+7$ of $\mathbb{Z}_{12}$ sends interlocking major/minor triads to interlocking ``jet/shark'' trichords, as in Figure~15 of the paper of Fiore--Noll--Satyendra \cite{fiorenollsatyendraSchoenberg}. The outer ring of Figure~\ref{fig:Schoenberg_Cycles} is mapped to its $(2,1,5)$-analogue via the affine map $x \mapsto 7x+7$ in such a way that each of the relevant squares commutes, see Figure~15 of \cite{fiorenollsatyendraSchoenberg}. The present paper contributes the observation that both the consonant $(1,6,10)$-cycle and the non-consonant $(2,1,5)$-cycle are labelled by the single transformation $(13)V$ and this transformation commutes with $x \mapsto 7x+7$ by Proposition~\ref{prop:Centralizer_of_J_in_Aff(3,Z12)}. Moreover, thanks to our decomposition of the uniform triadic transformations representation $\calh$ in Proposition~\ref{prop:H_decomposition}~\ref{prop:H_decomposition:iii}, we know that the instantiation $(13)V$ of RICH is an element of $\calh$, {\it for all 3-tuples}. Notice that not only do we have an economy of description in the sense of only a single transformation $(13)V$ for {\it both} the consonant $(1,6,10)$-cycle and the non-consonant $(2,1,5)$-cycle, but we also have the exact enchaining of consecutive chords, all in the representation $\calh$.

The musical interest in the image of the consonant $(1,6,10)$-cycle under the affine transformation $x \mapsto 7x+7$ is that the image provides pitch-class segment material for other passages in the piece, including the opening theme in measures 1 and 2. See Figure~3 of \cite{fiorenollsatyendraSchoenberg} for measure numbers of other passages with fragments of the affine image.

Proposition~\ref{prop:Centralizer_of_J_in_Aff(3,Z12)} applies not only to invertible affine transformations like $x \mapsto 7x+7$ in the preceding paragraphs, but also to non-invertible affine transformations such as $x \mapsto 10x$ in Figure~15 of \cite{fiorenollsatyendraSchoenberg}. Thus, Proposition~\ref{prop:Centralizer_of_J_in_Aff(3,Z12)} provides the mathematical justification for such instances of morphisms of simply transitive groups actions and morphisms of their associated {\it generalized interval systems}. See Section 2 of \cite{fiorenollsatyendraSchoenberg} for a development of such morphisms. Another example of a morphism is $x-2$ in Figure~\ref{fig:redo_of_Straus_on_Webern} of the Webern {\it Problembeispiel}.

Section~\ref{subsec:SchillingerNetwork} is dedicated to an example where these ideas are applied to the space $\mathbb{Z}_{7}^{\times 3}$ of generic scale degree triples rather than the space $\mathbb{Z}_{12}^{\times 3}$ of specific pitch class triples. Section \ref{subsec:Diatonic_Falling_Fifth} introduces this switch from specific to generic coordinates on the basis of yet another musically interesting example.

\subsection{Diatonic Falling Fifth Sequence as a Cycle} \label{subsec:Diatonic_Falling_Fifth} \leavevmode \smallskip

The Structure Theorem~\ref{thm:structure_of_J} and this entire paper are formulated for $\mathbb{Z}_{12}$ because of the main application to the twelve tone system. However, we did not use any specifics about 12, so analogous results also hold $\mathbb{Z}_n$. Particularly interesting is $n=7$ because $\mathbb{Z}_{7}$ models the diatonic pitch collection. We encode the underlying set of the $C$-major scale as $C \leftrightarrow 0$, $D \leftrightarrow 1$, $E \leftrightarrow 2$, \dots, and finally $B \leftrightarrow 6$. The diatonic falling fifth sequence is then encoded as in Figure~\ref{fig:Diatonic_Falling_Fifths}. To find a linear transformation with this orbit, we recall Theorem~\ref{thm:structure_of_J}~\ref{item:action_of_normal_form} and use the first three sequence elements to produce a system of 2 equations in 2 unkowns.
$$U(0,2,4)+m(4-0)+n(4-2)=(0,5,3) \hspace{.5in} U(5,0,3)+m(3-5)+n(3-0)=(1,6,3)$$
Notice that the output chords are in the $T/I$-class of the input, but need to be reordered by permutation (12) to match the sequence. The solution is $m=3$ and $n=0$, so the linear transformation in $\Sigma_3 \ltimes \calj(\mathbb{Z}_7)$ with orbit the diatonic falling fifth sequence is
$$(12)U(UV)^3=
\left(
\begin{array}{ccc}
5 & 0 & 3 \\
4 & 1 & 3 \\
5 & 1 & 2
\end{array}
\right).$$

\begin{figure}[h]
  \centering
  \includegraphics[width=5.5in]{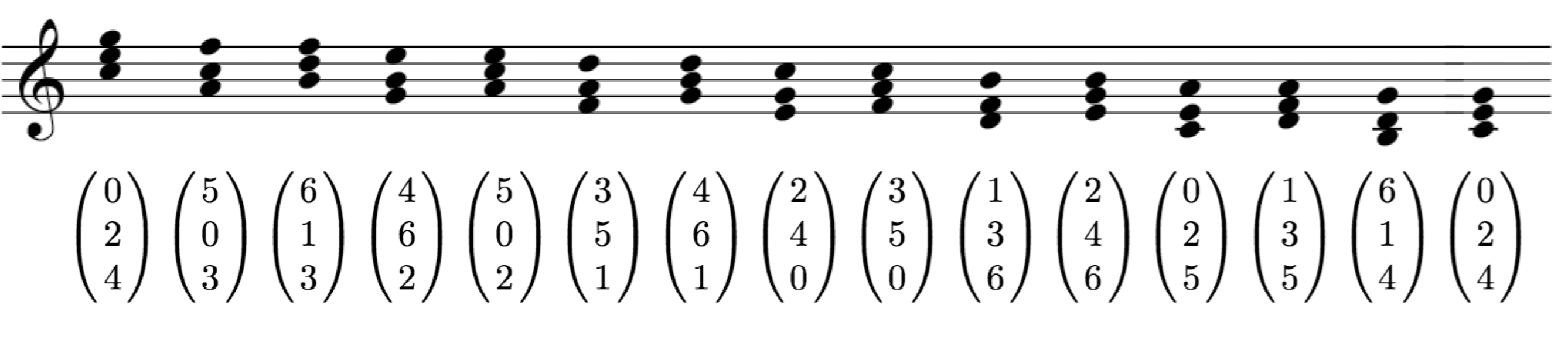}\\
  \caption{The diatonic falling fifth sequence and its encoding as vectors with entries in $\mathbb{Z}_7$. The encoding of the $C$-major scale collection is $C \leftrightarrow 0$, $D \leftrightarrow 1$, $E \leftrightarrow 2$, \dots, and finally $B \leftrightarrow 6$.
  The diatonic falling fifth sequence is a cycle of the sole transformation $(12)U(UV)^3$ in $\Sigma_3 \ltimes \calj(\mathbb{Z}_7)$. } \label{fig:Diatonic_Falling_Fifths}
\end{figure}

\subsection{A Transformational Idea of Joseph Schillinger Revisited} \label{subsec:SchillingerNetwork} \leavevmode \smallskip

The composer and teacher Joseph Schillinger (1895--1943) is one of the early pioneers of transformational thinking in musical composition. He utilized affine transformations in pitch and rhythm. And in particular he proposed the expansion and contraction of musical pitch, both in the specific pitch class domain $\mathbb{Z}_{12}$ as well as in the generic scale degree domain $\mathbb{Z}_7$. In Book 2: {\it Theory of Pitch Scales} (p. 138) of the posthumously published {\it Schillinger System of Musical Composition} \cite{Schillinger1946} there is a discussion of melodies, which -- from a transformational perspective -- can be viewed as results of a first expansion, i.e. an augmentation by factor $2 \, \, mod \, \, 7$. Schillinger argues that the (re-)contraction of the melody may provide new material while it preserves thematic continuity at the same time. He recommends to utilize the transformed thematic motives in the introduction or in interludes.

As an aside he mentions that {\it the processes of expanding and contracting music often leads to startling discoveries} and illustrates this statement with a comparison of Vincent Youmann's {\it Without a Song} and Nicolai Rimsky-Korsakov's {\it Hymn to the Sun} (from Act II of his opera {\it Coq d'Or}). A slight elaboration of this observation is presented in Figure~\ref{fig:SchillingerDiagram}.

\begin{sidewaysfigure}[h]
\vspace{5.75in}

\begin{subfigure}{1\textwidth}
\centering
\includegraphics[width=8.5in]{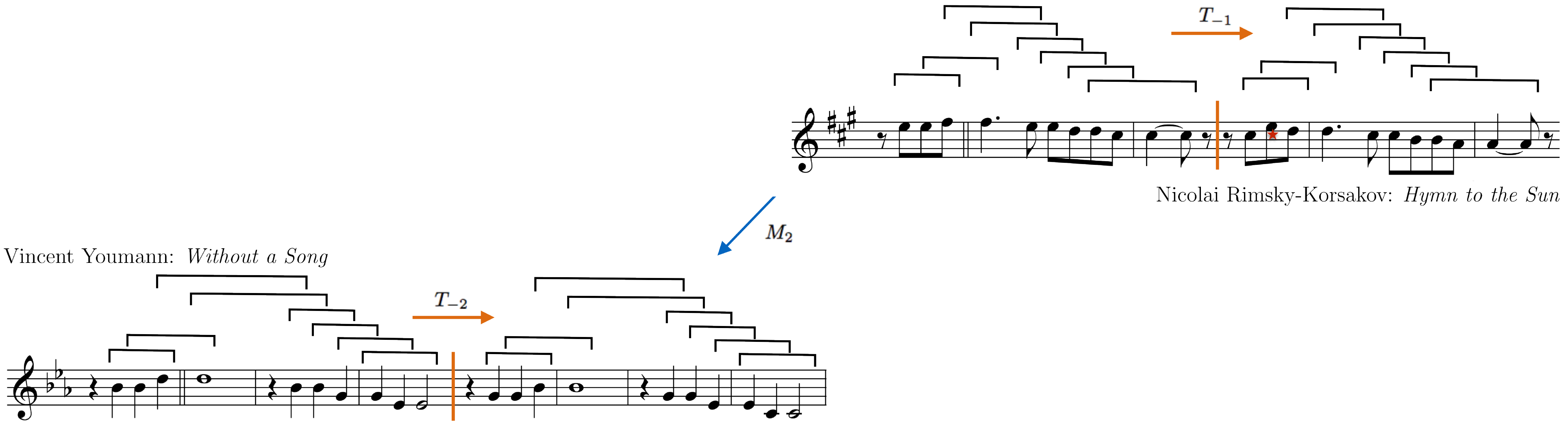}
\end{subfigure}
\vspace{.25in}

\begin{subfigure}{1\textwidth}
\centering
\includegraphics[width=8.5in]{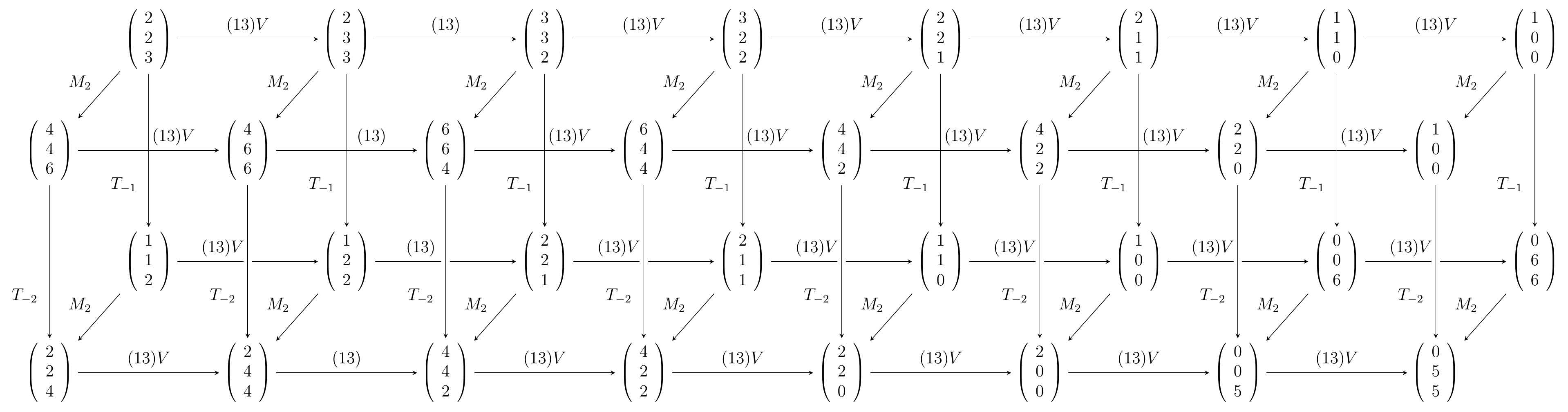}
\end{subfigure}

\caption{Transformational network inspired by an argument of Joseph Schillinger on melodic expansion. The two back rows of the diagram are the bracketed 3-note segments of Rimsky-Korsakov's {\it Hymn to the Sun}. In this {\it mod} 7 diagram, 0 corresponds to the last note of the respective initial melody, so $0=C\sharp$ in the back two rows, and (2,2,3) in the top back row corresponds to $E$, $E$, $F\sharp$. The front two rows are the bracketed 3-note segments of Youmann's {\it Without a Song}, so $0=E\flat$ there, and $(4,4,6)$ in the top front row corresponds to $B\flat$, $B\flat$, $D$. The transformation $M_2$ is multiplication by 2, and is indicated between the staves in blue at an angle to match its appearance in the diagram. The horizontal squares commute by Proposition~\ref{prop:Centralizer_of_J_in_Aff(3,Z12)}.  } \label{fig:SchillingerDiagram}

\end{sidewaysfigure}

Both melodies possess a model-sequence structure. In the {\it Hymn to the Sun} the sequence repeats the model one scale degree lower. It deviates in a single note (indicated by a star) from the model.
In {\it Without a Song} the sequence repeats the model two scale degrees lower. In order to illustrate the match of both melodies up to augmentation/contraction we have encoded their scale degrees in such a way that the last note of the model corresponds to the scale degree $0 \, \, mod \, \, 7$ in both cases, i.e. $C\sharp$ for {\it Hymn to the Sun} and $E\flat$ for {\it Without a Song}.
In the spirit of the present article we added a transformational analysis of the 10-note melodic models themselves.  These  melodies have been covered by all consecutive 3-note segments, which we regard as elements of $\mathbb{Z}_7^{\times 3}$. Both melodies exemplify the same transformational structure: except for the permutation $(13)$ between segments 2 and 3 there is the RICH-transformation which connects all other consecutive pairs of segments:
$$(13)V = \left(
\begin{array}{ccc}
0 & 1 & 0 \\
0 & 0 & 1 \\
-1 & 1 & 1
\end{array}
\right).$$
Note, that in contrast to Section \ref{subsec:Recalcitrant Viola} the RICH transform $(13)V$ in the present example acts on triples of scale degrees rather than triples of pitch classes.

\FloatBarrier

\subsection{Recovering $PLR$-$T/I$-Duality and a Special Case of Duality Theorem of Fiore--Satyendra} \label{subsec:Recovering_PLR-TI_Duality_and_FS} \leavevmode \smallskip

The centralizer results of Sections~\ref{subsec:JCentralizer_in_GL} and \ref{subsec:JCentralizer_in_Aff} and the normal form action in Theorem~\ref{thm:structure_of_J}~\ref{item:action_of_normal_form} allow us to recover the classical $PLR$-$T/I$-duality recalled in Section~\ref{sec:Recollection_of_PLR} and duality for trichords containing a generator for $\mathbb{Z}_{12}$ (a special case of an earlier result of Fiore--Satyendra \cite{fioresatyendra2005}).

\begin{thm} \label{thm:Recovering_Duality}
Suppose $(x,y,z)\in \mathbb{Z}_{12}^{\times 3}$ is such that $z-x$ is a generator of $\mathbb{Z}_{12}$. Then the restrictions of $\langle U,\; UV \rangle$ and $T/I$ to the $T/I$-orbit of $(x,y,z)$ are a Lewin dual pair of groups. That is, they each act simply transitively on the orbit and are centralizers of each other in the respective symmetric group.

If instead of $z-x$, the difference $z-y$ is a generator, then we have a similar statement for $\langle U, UW \rangle$ in place of $\langle U,\; UV \rangle$.
\end{thm}
\begin{proof}
Since $z-x$ is a generator of $\mathbb{Z}_{12}$, the $T/I$-orbit has 24 elements and simple transitivity of the 24-element $T/I$-group follows from the Orbit-Stabilizer Theorem. From Structure Theorem~\ref{thm:structure_of_J}~\ref{item:action_of_normal_form} we see that $\langle U,\; UV \rangle$ also acts transitively on the $T/I$-orbit and has order 24, so $\langle U,\; UV \rangle$ also acts simply transitively. By Propositions~\ref{prop:Centralizer_of_J_in_GL3} and \ref{prop:Centralizer_of_J_in_Aff(3,Z12)}, $\langle U,\; UV \rangle$ and $T/I$ commute. By Proposition~3.8 of \cite{BerryFiore}, simple transitivity on a finite set and commutativity together imply that the groups centralize one another.
\end{proof}

\begin{examp} \leavevmode \smallskip
\begin{enumerate}
\item
If we take $(x,y,z)=(0,4,7)$ in Theorem~\ref{thm:Recovering_Duality}, then the restriction of $\langle U,\; UV \rangle$ is the $PLR$-group as recalled in Section~\ref{sec:Recollection_of_PLR} and Theorem~\ref{thm:Recovering_Duality} recovers the familiar $PLR$-$T/I$-Duality. The transformation $(UV)^m$ is the {\it Schritt} $Q_{7m}$ and the transformation $U(UV)^m$ is a {\it Wechsel}.
\item
If we take $(x,y,z)=(0,4,1)$, then we obtain the {\it jet-shark} dual pair, and the Sub Dual Group Theorem of Fiore--Noll \cite{fiorenoll2011} applies to construct the jet-shark sub dual pair in the analysis of Schoenberg, String Quartet in $D$ minor, op. 7, in \cite{fiorenollsatyendraSchoenberg}.
\item
The hypothesis that $z-x$ is a generator in Theorem~\ref{thm:Recovering_Duality} is necessary. A counterexample is presented by $(x,y,z)=(0,4,10)$, which also plays a role in the Schoenberg analysis of \cite{fiorenollsatyendraSchoenberg}. In the case of $(x,y,z)=(0,4,10)$, the restriction of $\langle U,\; UV \rangle$ to the $T/I$-orbit of $(0,4,10)$ does not have 24 elements, namely the restrictions of $UV$ and $UV^7$ are the same. We have
$$UV(0,4,10)=(0,4,10)+10=(0,4,10)+7\cdot 10=UV^7(0,4,10)$$
$$UV(4,0,6)=(4,0,6)+2=(4,0,6)+7\cdot 2=UV^7(4,0,6)$$
and similarly on the translates. Notice also that the restriction of $\langle U,\; UV \rangle$ cannot possibly act transitively, as it preserves odd and even.
\item
Theorem~\ref{thm:Recovering_Duality} is a special case of Theorem~7.1 of Fiore--Satyendra \cite{fioresatyendra2005}, which treats $n$-tuples in $\mathbb{Z}_m$ that satisfy a ``tritone condition.''
\end{enumerate}
\end{examp}

\section{A Linear Representation of Uniform Triadic Transformations} \label{sec:A_Linear_Rep_of_the_UTTs}
In this section we construct a linear representation $\rho\co \calu \to GL(3, \mathbb{Z}_{12})$ of Julian Hook's group $\calu$ of uniform triadic transformations, characterize the image as the subgroup $\calh$ of $\Sigma_3 \ltimes \calj $ which fixes the set of triads in (non-dualistic) root position, and determine the structure of $\calh$, along with two normal forms of its elements. Hook's theory of uniform triadic transformations \cite{hookUTT2002} is an important chapter of well-established transformational music theory.

The representation $\rho$ sends transformations of {\it abstract triads} to linear transformations of root position triadic voicings. In Section~\ref{subsec:root_pos_and_dualisitic_root_pos}, we clarify the difference between {\it non-dualistic} root position and {\it dualistic} root position, while in Section~\ref{subsec:UTT_review} we recall how Hook treats the consonant triads as abstract entities, parametrized by their roots and their mode.  Importantly, we also recall how he studies certain rigid transformations of abstract consonant triads called {\it uniform triadic transformations}, which together form his group denoted $\calu$. The representation $\rho\co \mathcal{U} \to GL(3,\mathbb{Z}_{12})$ is defined in Section~\ref{subsec:defn_of_rho}: the consonant triads are represented in terms of their root position voicings $(x, y, z) \in \mathbb{Z}_{12}^{\times 3}$ and the uniform triadic transformations are realized as linear transformations that act on (non-dualistic) root position voicings exactly as they would on abstract consonant triads. Properties of the representation $\rho$ are investigated in Sections~\ref{subsec:Hook_Group_And_Its_Structure} and \ref{subsec:Hook_Group_and_UTTs}. As a final result we will prove in Section~\ref{subsec:Hook_Group_and_UTTs} that $\rho(\calu)=\calh$ is the wreath product $\Sigma_2 \lbag \mathbb{Z}_{12}$ and we will select new generators for $\calh$ that make $\rho$ apparent.

Most of the time, we will investigate the structure of $\calh$ as a subgroup of $\Sigma_3 \ltimes \calj $, instead of via $\calu$ and the homomorphism $\rho$.
Working with an internal perspective of the group $\Sigma_3 \ltimes \calj$ provides us with more flexibility in the conceptualization of these transformations, and provides proofs independent of knowledge of $\calu$. Namely, all the results rely on Structure Theorem~\ref{thm:structure_of_J} about $\calj$.

\subsection{Root Position and Dualistic Root Position} \label{subsec:root_pos_and_dualisitic_root_pos} \leavevmode \smallskip

Some explanation of the difference between non-dualistic root position and dualistic root position is in order. We use the unmodified term {\it root position} to mean {\it non-dualistic root position} (hence the optional modifier ``non-dualistic'' in parentheses in the foregoing paragraphs). The triple $(0, 4, 7) \in \mathbb{Z}_{12}^{\times 3}$ represents the {\it root position $C$-major triad}, which on the musical staff has tone $C$ in the lowest position, then the very next $E$ tone above that, and the very next $G$ tone above that. The vector $(0, 3, 7) \in \mathbb{Z}_{12}^{\times 3}$ represents the {\it root position $c$-minor triad}, which on the musical staff has tone $C$ in the lowest position, then the very next $E\flat$ tone above that, and the very next $G$ tone above that. Unfortunately, when written as column vectors $\left(\begin{array}{c} 0 \\ 4 \\ 7 \end{array} \right)$ and $\left(\begin{array}{c} 0 \\ 3 \\ 7 \end{array} \right)$, these root position triad representations appear in the opposite orientation of actual root position chords on the musical staff. Nevertheless, we prefer to use this encoding, despite its opposite orientation, because it is compatible with the mathematical transpose of vectors and because we find it more natural to have the root listed as the first note we read. {\it Root position major triads} are of the form $(r,r+4,r+7)$ and {\it root position minor triads} are of the form $(r,r+3,r+7)$.

Dualistic root position, on the other hand, reverses {\it in the minor triads only} the positions of the root and fifth in its notation. For instance, the triple $(0, 4, 7) \in \mathbb{Z}_{12}^{\times 3}$ represents both the {\it dualistic root position $C$-major triad} and the {\it root position $C$-major triad}, while the vector $(7, 3, 0) \in \mathbb{Z}_{12}^{\times 3}$ represents the {\it dualistic root position $c$-minor triad}. The encoding $(7, 3, 0)$ means that we read the root position $c$-minor on the staff from top note to bottom note, in other words, $(7, 3, 0)$ signifies literally the same musical chord as the root position encoding $(0, 3, 7)$: on the musical staff it has tone $C$ in the lowest position, then the very next $E\flat$ tone above that, and the very next $G$ tone above that. It does {\it not} mean a reordering of the actual tones when we say {\it dualistic root position} (the reordered notation could also signify reordered tones, as we do elsewhere in the article without the term dualistic root position).  {\it Dualistic root position major triads} are of the form $(r,r+4,r+7)$ and {\it dualistic root position minor triads} are of the form $(r+7,r+3,r)$.

The $\mathcal{J}$ orbit of $(0, 4, 7)$ is the set of all 24 dualistic root position triads, denoted by \texttt{DualRootPosTriads}. It is easy to see that the group $\mathcal{J}$ is the set stabilizer of \texttt{DualRootPosTriads} within the group $\Sigma_3 \ltimes \calj$. In Sections~\ref{subsec:Hook_Group_And_Its_Structure} and \ref{subsec:Hook_Group_and_UTTs} we investigate the analogous situation for the root position triads: $\calh$ is by definition the set-wise stabilizer in $\Sigma_3 \ltimes \calj$ of the root position triads, and the $\calh$ orbit of $(0, 4, 7)$ is the set of all 24 root position triads.

\subsection{Uniform Triadic Transformations of Abstract Triads} \label{subsec:UTT_review} \leavevmode \smallskip

We now revisit some anchor points from \cite{hookUTT2002} and stay close to the notations introduced there. Hook encodes a consonant triad abstractly as a pair $(r,\mu)$ consisting of its (non-dualistic) root pitch class $r \in \mathbb{Z}_{12}$ and a parity $\mu \in \{+,-\}$ to indicate major or minor. For instance,
$$\begin{array}{rrrr}
C = (0, +), \;\; & C \sharp = (1, +), \;\;& \dots,\;\; & B = (11, +), \\ c = (0, -), \;\; & c\sharp = (1, -), \;\; & \dots,\;\; & b = (11, -).\end{array}$$
Hook denotes the set of these abstractly represented consonant triads as
$$\Gamma:= \mathbb{Z}_{12} \times \{+, -\}.$$

A {\it uniform triadic transformation} is a function $\langle s, m, n \rangle\colon \Gamma \to \Gamma$ which translates the roots of input major chords by $m$, translates the roots of input minor chords by $n$, and preserves (respectively reverses) the input parity when the sign $s$ is $+$ (respectively $-$). Formally, we have:
$$\langle +, m, n \rangle(r,+)= (r+m,+) \hspace{1in} \langle +, m, n \rangle(r,-)= (r+n,-)\phantom{.}$$
$$\langle -, m, n \rangle(r,+)= (r+m,-) \hspace{1in} \langle -, m, n \rangle(r,-)= (r+n,+).$$

Hook denotes by $\mathcal{U}$ the {\it group of uniform triadic transformations}.
\begin{equation} \label{equ:U_definition}
\mathcal{U} := \left \{ \left < s, m, n \right > \, | \, s \in \{ +, - \} \text{ and } m, n \in \mathbb{Z}_{12}  \right \}
\end{equation}
The group operation is function composition and satisfies the following formulas.\footnote{We are using the usual order of function composition in this article, namely $g \circ f$ means to do $f$ first. In \cite{hookUTT2002}, Hook uses the opposite convention, so our equivalent formulas are slightly different.}
$$\langle +,p,q \rangle\circ\langle +, m, n \rangle = \langle +,m+p,n+q \rangle \hspace{.5in}
\langle +,p,q \rangle\circ\langle -, m, n \rangle=\langle -, m+q, n+p \rangle$$
$$\langle -,p,q \rangle\circ\langle +, m, n \rangle=\langle -,m+p,n+q \rangle \hspace{.5in}
\langle -,p,q \rangle\circ\langle -, m, n \rangle= \langle +,m+q,n+p\rangle$$
From these formulas we see that $\langle +, 1,0\rangle$ and $\langle +, 0, 1\rangle$ generate two commuting copies of $\mathbb{Z}_{12}$, and $\langle -,0,0 \rangle$ generates a (multiplicative) copy of $\mathbb{Z}_2$, and that conjugation
$$\langle -,0,0 \rangle^{-1} \circ \langle +,p,q \rangle \circ \langle -,0,0 \rangle=\langle +,q,p \rangle$$
interchanges the two copies of $\mathbb{Z}_{12}$. Hence, $\calu$ is isomorphic to the semi-direct product $\mathbb{Z}_{2} \ltimes (\mathbb{Z}_{12} \times \mathbb{Z}_{12})$ where $\mathbb{Z}_{2}$ permutes the two copies of $\mathbb{Z}_{12}$. In other words, $\mathcal{U}$ is isomorphic to the wreath product $\mathbb{Z}_{12} \lbag\mathbb{Z}_{2}$. From these isomorphisms, or even directly from equation \eqref{equ:U_definition}, we see that the total number of uniform triadic transformations is $\vert\mathcal{U}\vert=288 = 2 \cdot 12 \cdot 12$.

\subsection{A Representation of Uniform Triadic Transformations} \label{subsec:defn_of_rho} \leavevmode \smallskip

We next derive a representation $\rho\co \calu \to GL(3,\mathbb{Z}_{12})$ which acts on major and minor triads {\it in root position} exactly as $\calu$ does on abstract triads. First notice that the root position triads $C$-major, $c$-minor, and $C\sharp$-major
\begin{equation} \label{equ:root_position_basis}
\left(\begin{array}{c} 0 \\ 4 \\ 7 \end{array} \right) \hspace{.25in}
\left(\begin{array}{c} 0 \\ 3 \\ 7 \end{array} \right) \hspace{.25in}
\left(\begin{array}{c} 1 \\ 5 \\ 8 \end{array} \right)
\end{equation}
are clearly linearly independent over $\mathbb{Z}_{12}$. They also generate $\mathbb{Z}_{12}^{\times 3}$, since the standard basis can be expressed as linear combinations of them.
\begin{equation} \label{equ:expressing_standard_basis_via_chords}
\aligned
\left(\begin{array}{c} 1 \\ 0 \\ 0 \end{array} \right) &= 7\left(\begin{array}{c} 0 \\ 4 \\ 7 \end{array} \right)+
9\left(\begin{array}{c} 0 \\ 3 \\ 7 \end{array} \right) + \left(\begin{array}{c} 1 \\ 5 \\ 8 \end{array} \right) \\
\left(\begin{array}{c} 0 \\ 1 \\ 0 \end{array} \right) &= \phantom{7}\left(\begin{array}{c} 0 \\ 4 \\ 7 \end{array} \right)-
\phantom{9}\left(\begin{array}{c} 0 \\ 3 \\ 7 \end{array} \right)  \\
\left(\begin{array}{c} 0 \\ 0 \\ 1 \end{array} \right) &= 3\left(\begin{array}{c} 0 \\ 4 \\ 7 \end{array} \right)+
4\left(\begin{array}{c} 0 \\ 3 \\ 7 \end{array} \right)
\endaligned
\end{equation}
Since \eqref{equ:root_position_basis} is a basis, any linear transformation is uniquely determined by its action on these vectors.

In particular, $\rho\langle +, m, n\rangle$ is uniquely determined by the requirement that it adds $m$ to the $C$-major chord, adds $n$ to the $c$-minor chord, and adds $m$ to the $C\sharp$-major chord. Evaluation of $\rho\langle +, m, n \rangle$ on the standard basis via the formulas \eqref{equ:expressing_standard_basis_via_chords} computes the columns as follows.
\begin{equation} \label{equ:rho_def_part_1}
\begin{array}{lll}
\rho\left<+,m,n \right> & = & \left( \begin{array}{rrr} 1 - 4 m - 3 n & m - n & 3 m + 4 n  \\ - 4 m -3 n  & 1 + m - n & 3 m + 4 n  \\ - 4 m -3 n & m - n & 1 + 3 m + 4 n  \end{array} \right )
\end{array}
\end{equation}
Fortunately, $\rho\left<+,m,n \right>$ behaves as expected on the other root position triads, as we verify with direct computation.
\begin{equation} \label{equ:rho_does_as_should_on_+}
\small
\rho\left<+,m,n \right>\cdot \left(\begin{array}{c} r \\ r+4 \\ r+7\end{array} \right) = \left(\begin{array}{c} r+m \\ r+4+m \\ r+7+m\end{array}  \right), \hspace{.1in}
\rho\left<+,m,n \right>\cdot \left(\begin{array}{c} r \\ r+3 \\ r+7\end{array} \right) = \left(\begin{array}{c} r+n \\ r+3+n \\ r+7+n \end{array}  \right)
\end{equation}

For the mode-reversing transformations, $\rho\langle -, m, n\rangle$ is uniquely determined by the requirement that it adds $m$ to the $C$-major chord, adds $n$ to the $c$-minor chord, and adds $m$ to the $C\sharp$-major chord, {\it and then reverses the mode} for all three. Evaluation of $\rho\langle -, m, n \rangle$ on the standard basis via the formulas \eqref{equ:expressing_standard_basis_via_chords} computes the columns as follows.
\begin{equation} \label{equ:rho_def_part_2}
\begin{array}{lll}
\rho\left<-,m,n \right> & = & \left( \begin{array}{rrr} 1 - 3 m - 4 n & - m + n & 4 m + 3 n  \\  1 - 3 m - 4 n  & -1 - m + n  & 1 + 4 m + 3 n  \\ - 3 m -4 n & - m + n & 1 + 4 m + 3 n  \end{array} \right )
\end{array}
\end{equation}
Fortunately, $\rho\left<-,m,n \right>$ behaves as expected on the other root position triads, as we verify with direct computation.
\begin{equation} \label{equ:rho_does_as_should_on_-}
\small
\rho\left<-,m,n \right>\cdot \left(\begin{array}{c} r \\ r+4 \\ r+7\end{array} \right) = \left(\begin{array}{c} r+m \\ r+3+m \\ r+7+m\end{array}  \right), \hspace{.1in}
\rho\left<-,m,n \right>\cdot \left(\begin{array}{c} r \\ r+3 \\ r+7\end{array} \right) = \left(\begin{array}{c} r+n \\ r+4+n \\ n+r+7\end{array}  \right)
\end{equation}

The representation $\rho\co \calu \to GL(3,\mathbb{Z}_{12})$ is defined by the formulas \eqref{equ:rho_def_part_1} and \eqref{equ:rho_def_part_2}. The map $\rho$ is an injection because the image elements are all different on the basis \eqref{equ:root_position_basis}, as we see from \eqref{equ:rho_does_as_should_on_+} and \eqref{equ:rho_does_as_should_on_-} without any more computation. The map $\rho$ is a homomorphism because $\rho$ acts on major and minor
triads in root position exactly as $\calu$ does on abstract triads, and because the set of major and minor triads in root position contains the basis \eqref{equ:root_position_basis} (forming matrix representations with respect to selected bases is a group homomorphism).

\subsection{The Hook Group and its Structure} \label{subsec:Hook_Group_And_Its_Structure} \leavevmode \smallskip

The {\it Hook group} $\calh$ is the subgroup of $\Sigma_3 \ltimes \calj$ that maps {\it root position} triads to {\it root position} triads. It is analogous to $\calj$ in that $\calj$ is the subgroup of $\Sigma_3 \ltimes \calj$ that maps {\it dualistic root position} triads to {\it dualistic root position} triads. In other words, $\calh$ is the set-wise stabilizer in $\Sigma_3 \ltimes \calj$ of \texttt{RootPosTriads}, while $\calj$ is the set-wise stabilizer in $\Sigma_3 \ltimes \calj$ of \texttt{DualRootPosTriads}, recall the notation and discussion from Section~\ref{subsec:Triadic_Orbits}. For the major triads, root position and dualistic root position are the same. However, for the minor triads, root position and dualistic root position are related via the permutation $(1\;3)$. From this relationship, we anticipate $\calh=\calj^+ \; \bigsqcup \; (1\;3)\,\calj^-$, similar to $\calj=\calj^+ \; \bigsqcup \; \calj^-$.

\begin{prop}[Decomposition of $\calh$ and its Normal Form] \label{prop:H_decomposition} \leavevmode
\begin{enumerate}
\item \label{prop:H_decomposition:i}
The subgroup of mode-preserving operations in $\calh$, denoted $\calh^+:=\calh \cap \Sigma_3\calj^+$, is equal to $\calj^+$. \item \label{prop:H_decomposition:ii}
The subgroup of mode-reversing operations in $\calh$, denoted $\calh^-:=\calh \cap \Sigma_3\calj^-$, is equal to $(1\;3)\calj^-$.
\item \label{prop:H_decomposition:iii}
The group $\mathcal{H}$ has $288$ elements and consists of the $144$ mode-preserving transformations $(UV)^m(UW)^n$ and the $144$ mode-reversing transformations \\ $(1\;3) U(UV)^m(UW)^n$.
$$\calh=\calj^+ \; \bigsqcup \; (1\;3)\,\calj^-$$
\end{enumerate}
\end{prop}
\begin{proof}
Recall from Theorem~\ref{thm:structure_of_J}~\ref{item:action_of_normal_form} that $(UV)^m(UW)^n$ adds a constant to an input, so that $(UV)^m(UW)^n$ preserves both triad mode {\it and} any triad position. The nontrivial permutations in $\Sigma_3$, on the other hand, preserve triad mode but nontrivially change triad position.
\begin{enumerate}
\item
From the foregoing, an arbitrary element $\sigma(UV)^m(UW)^n \in \Sigma_3\calj^+$ is in $\calh$ if and only if $\sigma$ is the identity.
\item
Recall that $U$ reverses triad mode but preserves {\it dualistic} root position, so the second formula of Theorem~\ref{thm:structure_of_J}~\ref{item:action_of_normal_form} shows that $U(UV)^m(UW)^n$ reverses triad mode and preserves {\it dualistic} root position. Hence, an arbitrary element $\sigma U(UV)^m(UW)^n \in \Sigma_3\calj^-$ is in $\calh$ if and only if $\sigma$ is $(1\;3)$.
\item
This is an immediate consequence of \ref{prop:H_decomposition:i} and \ref{prop:H_decomposition:ii}. The set $(1\;3) \calj^-$ has 144 elements because we have left cancellation in any group.  $$(1\;3)j=(1\;3)j' \hspace{.5in} \Longrightarrow \hspace{.5in} j=j'$$
\end{enumerate}
\end{proof}

\begin{prop} \label{prop:H_two_generators}
The Hook group $\mathcal{H}$ is generated by the two elements $(1\;3)U$ and $(1\;3)W$.
\end{prop}
\begin{proof}
The two elements $(1\;3)U$ and $(1\;3)W$ are in $\calh$, since
$$(1\;3)U=(1\;3)U(UV)^0(UW)^0 \hspace{1in} (1\;3)W=(1\;3)U(UV)^0(UW)^1.$$

Conversely, any $(UV)^m(UW)^n\in \calh^+$ can be expressed in terms of $(1\;3)U$ and $(1\;3)W$ because
$$(1\;3)U\cdot(1\;3)U=VU=(UV)^{-1}$$
by Proposition~\ref{prop:Sigma3_acts_on_J}~\ref{prop:H_decomposition:ii}, and
$$(1\;3)W\cdot(1\;3)U=WU=(UW)^{-1}$$
also by Proposition~\ref{prop:Sigma3_acts_on_J}~\ref{prop:H_decomposition:ii}. Moreover, any $(1\;3) U(UV)^m(UW)^n \in \calh^-$, can be obtained from the previous two equations and $(1\;3) U$.

Thus, $\big\langle (1\;3)U, (1\;3)W \big\rangle = \calh$.
\end{proof}

\begin{rmk}[Matrices for Generators $(1\;3)U$ and $(1\;3)W$ of Hook Group $\mathcal{H}$]
The matrices $(1\;3)U$ and $(1\;3)W$ are obtained from the matrices $U$ and $W$ in Section~\ref{subsec:Definition_of_J}
by exchanging their top and bottom rows.
$$(1\;3)U = \left( \begin{array}{ccc} 1 & 1 & {-1} \\ 1 & 0 & 0 \\ 0 & 1 & 0 \end{array} \right ) \hspace{1in}
(1\;3)W = \left( \begin{array}{ccc} 1 & 0 & 0  \\ 1 & {-1} & 1 \\ 0 & 0 & 1 \end{array}  \right )$$
\end{rmk}

In preparation for two propositions about the structure of the Hook group, we inspect the orders of $(1\;3)U$, $(1\;3)W$, and  $(1\;3)W\cdot(1\;3)U$.

\begin{prop}[Orders of Generators $(1\;3)U$ and $(1\;3)W$ of Hook Group $\mathcal{H}$] \label{prop:orders_of_generators} \leavevmode
\begin{enumerate}
\item \label{prop:orders_of_generators:powers_of_(13)U}
The even powers of $(1\;3)U$ are
$$\big((1\;3)U\big)^{2m}=\Big((1\;3)U\cdot(1\;3)U\Big)^{m}=\big(VU\big)^{m}=\big(UV\big)^{-m},$$
while the odd powers are $$\big((1\;3)U\big)^{2m+1}=\big((1\;3)U\big)\big(UV\big)^{-m}.$$
Consequently, $\text{\rm order}\;(1\;3)U=24$.
\item
$\text{\rm order}\;(1\;3)W=2$.
\item \label{prop:orders_of_generators:product}
The product $(1\;3)W\cdot(1\;3)U$ equals $(UW)^{-1}$, so consequently \\
$\text{\rm order}\;\Big((1\;3)W\cdot(1\;3)U\Big)=12.$
\end{enumerate}
\end{prop}
\begin{proof} \leavevmode
\begin{enumerate}
\item
The even powers of $(1\;3)U$ are
$$\big((1\;3)U\big)^{2m}=\Big((1\;3)U\cdot(1\;3)U\Big)^{m}=\big(VU\big)^{m}=\big(UV\big)^{-m},$$
which are all distinct from one another for $m=0,1,\dots,11$ by Theorem~\ref{thm:structure_of_J}~\ref{thm:structure_of_J:semi-direct_product_form}. The odd powers of
$(1\;3)U$ are $$\big((1\;3)U\big)^{2m+1}=\big((1\;3)U\big)\big(UV\big)^{-m},$$
by the preceding statement; these are all distinct from one another for $m=0,1,\dots,11$ by left cancellation.
The odd powers are mode reversing while the even powers are mode preserving, so they are different, and $(1\;3)U$ has order 24.
\item
Using Proposition~\ref{prop:Sigma3_acts_on_J}~\ref{prop:Sigma3_acts_on_J:ii}, we have
$(1\;3)W\cdot (1\;3)W=(1\;3)W(1\;3)\cdot W =W \cdot W =\text{Id}$.
\item
Again using Proposition~\ref{prop:Sigma3_acts_on_J}~\ref{prop:Sigma3_acts_on_J:ii}, we have
$(1\;3)W\cdot(1\;3)U=(1\;3)W(1\;3)\cdot U=WU=(UW)^{-1}$, which is known to have order 12 by Theorem~\ref{thm:structure_of_J}~\ref{thm:structure_of_J:UV_and_UW_have_order_12}.
\end{enumerate}
\end{proof}


\begin{prop}[Another Normal Form for $\calh$] \label{prop:products_of_generators} \leavevmode
\begin{enumerate}
\item \label{prop:products_of_generators:intersection_trivial}
The group generated by the product $(1\;3)W\cdot(1\;3)U$  is equal to $\langle UW \rangle$, and consequently its intersection with $\big\langle (1\;3)U \big\rangle$ is trivial. \\
$$\big\langle (1\;3)W\cdot(1\;3)U \big \rangle \bigcap \big\langle (1\;3)U \big\rangle  = \{\text{\rm Id}\}$$
\item \label{prop:products_of_generators:uniqueness}
Every element of $\mathcal{H}$ can be written uniquely in the form $$\big((1\;3)U\big)^p (UW)^n$$ where $p=0,1,\dots,23$ and $n=0,1,\dots,11$.
\item
Such a product belongs to $\mathcal{H}^+$ (respectively $\mathcal{H}^-$) if and only if the exponent $p$ is even (respectively odd).
\end{enumerate}
\end{prop}
\begin{proof} \leavevmode
\begin{enumerate}
\item
From Proposition~\ref{prop:orders_of_generators}~\ref{prop:orders_of_generators:product}, the product $(1\;3)W\cdot(1\;3)U$ is the mode-preserving function $(UW)^{-1}$, so the product generates a subgroup $\langle UW \rangle$ of mode-preserving functions. From Proposition~\ref{prop:orders_of_generators}~\ref{prop:orders_of_generators:powers_of_(13)U}, the even powers of $(1\;3)U$ are $(UV)^{-m}$, which are not in  $\langle UW \rangle$. The odd powers of $(1\;3)U$ are mode-reversing, so also not in $\langle UW \rangle$.
\item
Uniqueness follows from \ref{prop:products_of_generators:intersection_trivial}.
For the existence, we use the decomposition $\calh=\calj^+ \; \bigsqcup \; (1\;3)\,\calj^-$, and recall $\calj^+= \langle UV,UW\rangle$ and $$(1\;3)\,\calj^- = (1\;3)\,U\calj^+=(1\;3)\,U\,\langle UV,UW\rangle.$$
All of these can be expressed in terms of powers of $(1\;3)\,U$ and $UW$ by Proposition~\ref{prop:orders_of_generators}~\ref{prop:orders_of_generators:powers_of_(13)U}.
\item
This follows from Proposition~\ref{prop:orders_of_generators}~\ref{prop:orders_of_generators:powers_of_(13)U} and that fact that composing a mode-preserving transformation with $U$ gives a mode-reversing transformation.
\end{enumerate}
\end{proof}

Note, that the cyclic groups $\langle (1\;3)U \rangle$ and $\langle UW \rangle$ do not commute. Hence, the Hook Group $\mathcal{H}$ is {\it not} their internal direct product.

The group structure of $\mathcal{H}$ is a semi-direct product of $\calj^+$ and the other generator $(1\;3)W$ as we see now. In Section~\ref{subsec:Hook_Group_and_UTTs}, after a ``change of basis'' in $\calj^+$, we will even see that $\calh$ is the wreath product $\mathbb{Z}_2 \lbag \mathbb{Z}_{12}$.

\begin{thm} \label{thm:H_is_semidirect_product}
The group $\calh$ is the semi-direct product $\langle (1\;3)W \rangle \ltimes \calj^+$. The conjugation action is
\begin{equation} \label{equ:conjugation_action_in_H}
(1\;3)W\cdot\Big( (UV)^m(UW)^n  \Big)\cdot W(1\;3) = (UV)^{m+n}(UW)^{-n}.
\end{equation}
\end{thm}
\begin{proof}
The subgroup $\calj^+ \leq \calh$ has index 2 and is thus normal. The order 2, mode-reversing transformation $(1\;3)W$ is not in $\calj^+$, so the other coset is $(1\;3)W\calj^+$, and $\langle (1\;3)W \rangle \calj^+=\calh$. Clearly, $\langle (1\;3)W \rangle \cap \calj^+=\{\text{Id}\}$, so finally $\calh = \langle (1\;3)W \rangle \ltimes \calj^+$.

The conjugation formula \eqref{equ:conjugation_action_in_H} follows from Theorem~\ref{thm:structure_of_J}~\ref{item:action_of_normal_form}, the definition of $W$, and the rewriting
$$m\big(z-x\big)+n\big(y-x\big)=(m+n)\big(z-x\big)+(-n)\big(z-y\big).$$
\end{proof}

\subsection{The Hook Group is our Linear Representation of the UTT Group} \label{subsec:Hook_Group_and_UTTs} \leavevmode \smallskip

Recall the representation $\rho\co \calu \to GL(3,\mathbb{Z}_{12})$ from Section~\ref{subsec:defn_of_rho} in equations \eqref{equ:rho_def_part_1} and \eqref{equ:rho_def_part_2}. This representation is an embedding, and acts on root position triads as it should, see equations \eqref{equ:rho_does_as_should_on_+} and \eqref{equ:rho_does_as_should_on_-}. Since $\rho(\calu)$ maps root position triads to root position triads, $\rho(\calu)$ is contained in $\calh$, the subgroup of $\Sigma_3 \ltimes \calj$ that preserves the set of root position triads. The cardinalities are $\vert\rho(\calu)\vert=288=\vert\calh\vert$, so the representation image $\rho(\calu)$ is equal to the Hook group $\calh$.
We would next like to ``change the basis'' of $\calj^+$ in $\calh$ to express $\calh$ as a wreath product $\mathbb{Z}_2 \lbag \mathbb{Z}_{12}$ and to express $\rho$ uniformly.

\begin{thm}[New Basis for $\calj^+$ to Make $\calh$ into a Wreath Product] \label{thm:towards_wreath_product} \leavevmode
\begin{enumerate}
\item \label{thm:towards_wreath_product:rho_+_1_0}
The transformation $(UV)^4(UW)^{-1}$ translates root position major triads by 1 and fixes root position minor triads, so $\rho\langle +,1,0 \rangle=(UV)^4(UW)^{-1}.$
\item \label{thm:towards_wreath_product:rho_+_0_1}
The transformation $(UV)^3(UW)^{1}$ fixes root position major triads and translates root position minor triads by 1, so $\rho\langle +,0,1 \rangle=(UV)^3(UW)^{1}.$
\item \label{thm:towards_wreath_product:generate_J+}
The two transformations $(UV)^4(UW)^{-1}$ and $(UV)^3(UW)^{1}$ generate $\calj^+$.
\item \label{thm:towards_wreath_product:rho_-_0_0}
The transformation $(1\;3)W$ switches parity of root position triads, without translation, so $\rho\langle -,0,0 \rangle=(1\;3)W$.
\item \label{thm:towards_wreath_product:(13)W_exchanges_new_generators}
Conjugation by $(1\;3)W$ exchanges these two new generators of $\calj^+$ in \ref{thm:towards_wreath_product:rho_+_1_0} and \ref{thm:towards_wreath_product:rho_+_0_1}.
\item
The group $\calh$ is isomorphic to the wreath product $\mathbb{Z}_2 \lbag \mathbb{Z}_{12}$, that is, the semi-direct product $\Sigma_2 \ltimes \big(\mathbb{Z}_{12} \times \mathbb{Z}_{12} \big)$ where $\Sigma_2$ acts by exchanging the two copies of $\mathbb{Z}_{12}$.
\end{enumerate}
\end{thm}
\begin{proof} \leavevmode
\begin{enumerate}
\item
To find $(UV)^m(UW)^{n}$ which translates root position major triads by 1 and fixes root position minor triads, we must solve in $\mathbb{Z}_{12}$ the system of equations
$$m7+n3=1 \hspace{1in} m7+n4=0,$$
obtained from evaluating Theorem~\ref{thm:structure_of_J}~\ref{item:action_of_normal_form} on root position major triads and root position minor triads $(x,y,z)$. The solutions are found to be $m=4$ and $n=-1$.
\item
Similarly, we solve
$$m7+n3=0 \hspace{1in} m7+n4=1$$
and find $m=3$ and $n=1$.
\item
The transformations $UV$ and $UW$ commute, so the product $$(UV)^4(UW)^{-1}\cdot (UV)^3(UW)^{1}$$ is $(UV)^7,$ which generates $\langle UV \rangle$, so that $UW=\big((UV)^7\big)^3\cdot(UV)^3(UW)^{1}$ can also be expressed in terms of $(UV)^4(UW)^{-1}$ and $(UV)^3(UW)^{1}$.
\item
Straightforward computation.
\item
Straightforward computation using the conjugation formula \eqref{equ:conjugation_action_in_H}.
\item
This is a consequence of Theorem~\ref{thm:H_is_semidirect_product} and parts \ref{thm:towards_wreath_product:generate_J+} and \ref{thm:towards_wreath_product:(13)W_exchanges_new_generators}.
\end{enumerate}
\end{proof}

\begin{cor}[Uniform Formula for $\rho$ and Semi-Direct Product Structure] \label{cor:New_Basis_Uniform_Formula}
Let us denote the new generators of $\calh$ by
$$E=(1\;3)W \hspace{.5in} F=(UV)^4(UW)^{-1} \hspace{.5in} G=(UV)^3(UW)^{1}.$$
Let $\text{\rm index}\co \{+,-\} \to \mathbb{Z}_2$ be the isomorphism from the multiplicative cyclic group of order 2 to the additive cyclic group of order 2, that is $\text{\rm index}(+) = 0, \;\text{\rm index}(-) = 1$.
The injective representation $\rho\co \calu \to GL(3,\mathbb{Z}_{12})$ from Section~\ref{subsec:defn_of_rho} in equations \eqref{equ:rho_def_part_1} and \eqref{equ:rho_def_part_2} is an isomorphism onto its image $\calh$ and satisfies
\begin{equation} \label{equ:formula_for_rho_in_new_generators}
\rho\langle s,m,n \rangle=E^{\text{\rm index}(s)}F^m G^n.
\end{equation}
Moreover, the group operation in $\calh$ in terms of these new generators is
\begin{equation} \label{equ:formula_for_product_in_new_generators}
E^t F^p G^q \cdot E^s F^m G^n = \begin{cases} E^t F^{m+p} G^{n+q} & \text{\rm if } s=0\phantom{0} \\ E^{t+1} F^{m+q} G^{n+p} & \text{\rm if } s=1.  \end{cases}
\end{equation}
\end{cor}
\begin{proof}
The formula for the homomorphism $\rho$ in \eqref{equ:formula_for_rho_in_new_generators} follows directly from Theorem~\ref{thm:towards_wreath_product}~\ref{thm:towards_wreath_product:rho_+_1_0}, \ref{thm:towards_wreath_product:rho_+_0_1}, and \ref{thm:towards_wreath_product:rho_-_0_0}.

The equation for the group structure in \eqref{equ:formula_for_product_in_new_generators} follows from the commutativity of $
F$ and $G$ and Theorem~\ref{thm:towards_wreath_product}~\ref{thm:towards_wreath_product:(13)W_exchanges_new_generators}. Equation \eqref{equ:formula_for_product_in_new_generators} also follows from the wreath product structure of $\calu$ in combination with the fact that $\rho$ is an isomorphism onto its image.
\end{proof}

\begin{rmk}[Matrix Representation for New Basis of $\calj^+$]
From the Theorem~\ref{thm:structure_of_J}~\ref{item:action_of_normal_form} we find matrix representations for the new basis of $\calj^+$.
$$\rho\langle +,1,0 \rangle=(UV)^4(UW)^{-1}=\left( \begin{array}{ccc} 9 & 1 & 3 \\ 8 & 2 & 3 \\ 8 & 1 & 4 \end{array} \right)$$
$$\rho\langle +,0,1 \rangle=(UV)^3(UW)^{1}=\left( \begin{array}{ccc} {10} & {11} & 4 \\ 9 & 0 & 4 \\  9 & {11} & 5 \end{array} \right)$$
\end{rmk}

\section{Conclusion} \label{sec:Conclusion}

To finish up, we may now return to our Webern {\it Problembeispiel} in Section~\ref{subsec:MotivationalProblembeispiel}. In Figure~\ref{fig:redo_of_Straus_on_Webern} we saw how to interpret the exact sequences in the Webern piece using the permutation $(12)$ and a voice reflection $V$. This motivated us to determine in Theorem~\ref{thm:structure_of_J} the structure of the group $\calj$ generated by the voice reflections $U$, $V$, and $W$, and to determine in Proposition~\ref{prop:Sigma3_J_is_SemiDirectProduct} the structure of $\langle \Sigma_3, \calj \rangle$. We further understood $\Sigma_3 \ltimes \calj$ in Section~\ref{subsec:Triadic_Orbits} through some special subgroups and their triadic orbits. Our determination of the centralizer of $\calj$ in $GL(3, \mathbb{Z}_{12})$ and $\text{Aff}(3, \mathbb{Z}_{12})$ in Propositions~\ref{prop:Centralizer_of_J_in_GL3} and \ref{prop:Centralizer_of_J_in_Aff(3,Z12)} guaranteed that $x-2$ in Figure~\ref{fig:redo_of_Straus_on_Webern} is a morphism of generalized interval systems. Further examples of morphisms are in Section~\ref{subsec:Affine_Morphisms}, and special cases of known duality theorems also follow from the aforementioned centralizers. Our group $\Sigma_3 \ltimes \calj$ allows us to sometimes turn alternating cycles into orbits of a single transformation, as we did with the Grail motive in Section~\ref{subsec:Grail_Motive}, some triadic networks in Schoenberg's String Quartet in $D$ minor, and the diatonic falling fifth sequence. As a final application of our Structure Theorem~\ref{thm:structure_of_J}, we characterize our representation of Hook's uniform triadic transformation group $\calu$ as the subgroup $\calh$ of $\Sigma_3 \ltimes \calj$ that maps root position consonant triads to root position consonant triads.

\subsection*{Acknowledgements}

Thomas Fiore gratefully acknowledges support from several sources during the genesis of this paper. A Humboldt Research Fellowship for Experienced Researchers supported Thomas Fiore during his 2015-2016 visit at Universit{\"a}t Regensburg. The Max-Planck-Institut f\"ur Mathematik supported his research stay in Bonn in July 2013. Earlier support in Michigan included a Rackham Faculty Research Grant of the University of Michigan and a Small Grant for Faculty Research from the University of Michigan-Dearborn. Thomas Fiore thanks the Regensburg {\it Sonderforschungsbereich 1085: Higher Invariants} and the Regensburg Mathematics Department for a very stimulating working environment.

Thomas Noll thanks the Max-Planck-Institut f\"ur Mathematik for the support of his research visit in Bonn in July 2013.


\end{document}